\documentclass{article}

\usepackage[utf8]{inputenc}

\usepackage{amsmath,amssymb,amsthm}
\numberwithin{equation}{section}

\usepackage{color}
\usepackage[
natbib=true,
maxcitenames=3, 
maxbibnames=10,  
sorting=ydnt,
sortcites,
backend=biber,
citestyle=authoryear,
bibstyle=authoryear,
uniquelist=false
]{biblatex}
\usepackage{hyperref}
\hypersetup{
	colorlinks = true,
    final=true,
    plainpages=false,
    pdfstartview=FitV,
    pdftoolbar=true,
    pdfmenubar=true,
    pdfencoding=auto,
    psdextra,
    bookmarksopen=true,
    bookmarksnumbered=true,
    breaklinks=true,
    linktocpage=true,
}
\def\setpdfinfo{
	\hypersetup{pdfinfo={
		Title={\@title},
		Author={\@author},
		Subject={\@thesistype},
		Keywords={\@keywords}}}
}

\usepackage{xcolor}
\definecolor{faunat}{RGB}{0,155,119}
\definecolor{faulnat}{RGB}{170,207,189}
\definecolor{faullnat}{RGB}{229,239,234}
\colorlet{cfaunat}{faunat>twheel,1,2}
\colorlet{teco}{faunat>twheel,6,12}

\hypersetup{urlcolor=faunat,citecolor=cfaunat,linkcolor=faunat}

\usepackage[nameinlink, capitalise, noabbrev]{cleveref}

\usepackage{caption}
\usepackage{subcaption}

\usepackage[page,toc,titletoc,title]{appendix}

\usepackage{graphicx}

\usepackage[margin=1.5in]{geometry}

\usepackage{todonotes}

\usepackage{wrapfig}

\newcommand{\R}{\mathbb{R}}
\newcommand{\N}{\mathbb{N}}
\newcommand{\calN}{\mathcal{N}}
\renewcommand{\H}{\mathcal{H}}
\newcommand{\X}{\mathcal{X}}

\renewcommand{\d}{\mathrm{d}}
\newcommand{\sg}{\zeta}
\newcommand{\func}{\mathcal{J}}

\newcommand{\dom}{\operatorname{dom}}
\newcommand{\tv}{\operatorname{TV}}
\DeclareMathOperator*{\esssup}{ess\,sup}
\DeclareMathOperator*{\argmin}{arg\,min}
\DeclareMathOperator*{\argmax}{arg\,max}
\newcommand{\dist}{\operatorname{dist}}
\newcommand{\norm}[1]{\left\|#1\right\|}
\newcommand{\abs}[1]{\left|#1\right|}

\newcommand{\st}{\,:\,}

\newcommand{\gammato}{\overset{\Gamma}{\to}}
\mathchardef\mhyphen="2D
\newcommand{\pprox}{\operatorname{prox}^p}
\newcommand{\prox}{\operatorname{prox}}

\newcommand{\tex}{T_\mathrm{ex}}
\renewcommand{\div}{\mathrm{div}}

\newtheorem{thm}{Theorem}[section]
\newtheorem{prop}{Proposition}[section]
\newtheorem{lemma}{Lemma}[section]

\theoremstyle{definition}
\newtheorem{definition}{Definition}[section]
\newtheorem{ass}{Assumption}
\newtheorem*{ass*}{Assumption}

\theoremstyle{remark}
\newtheorem{rem}{Remark}[section]
\newtheorem{example}{Example}[section]

\newcommand{\revise}[1]{{\color{orange}#1}}
\renewcommand{\revise}[1]{{#1}}

\title{{Gradient Flows and Nonlinear Power Methods for the Computation of Nonlinear Eigenfunctions}
}
\author{Leon Bungert\thanks{Hausdorff Center for Mathematics, University of Bonn, Endenicher Allee 62, Villa Maria,
53115 Bonn, Germany. Email: \href{mailto:leon.bungert@hcm.uni-bonn.de}{leon.bungert@hcm.uni-bonn.de}} \and Martin Burger\thanks{Friedrich-Alexander University Erlangen-Nürnberg, Department Mathematics, Cauerstr. 11, 91058 Erlangen, Germany. Email: \href{mailto:martin.burger@fau.de}{martin.burger@fau.de}}}
\date{\today}

\makeatletter
\let\blx@rerun@biber\relax
\makeatother

\begin{document}

\maketitle

\begin{abstract}
    This chapter describes how gradient flows and nonlinear power methods in Banach spaces can be used to solve nonlinear eigenvector-dependent eigenvalue problems, and how convergence of (discretized) approximations can be verified.
    We review several flows from literature, which were proposed to compute nonlinear eigenfunctions, and show that they all relate to normalized gradient flows.  
    Furthermore, we show that the implicit Euler discretization of gradient flows gives rise to a nonlinear power method of the proximal operator and prove their convergence to nonlinear eigenfunctions.
    Finally, we prove that $\Gamma$-convergence of functionals implies convergence of their ground states, which is important for discrete approximations.
\end{abstract}

\tableofcontents

\section{Introduction}

Nonlinear eigenvalue problems appear in different applications in physics, \citep{weinstein1982nonlinear,
cohen2018energy}, mathematics \citep{cances2010numerical,
rabinowitz1971some,
amann1976fixed}, and also in modern disciplines like data science \citep{hein2010inverse,
buhler2009spectral,bungert2019computing,
gilboa2018nonlinear}.
While linear eigenvalue problem have the form $\lambda u = A u$, where $A$ is a linear operator or a matrix, with \emph{nonlinear eigenvalue problems} we refer to equations of the form $\lambda B(u)=A(u)$ which depend non-linearly on the eigenvector.
For instance, solutions of such nonlinear eigenvalue problems can be used to describe non-Newtonian fluids or wave propagation in a nonlinear medium.
Only in the last years it was noticed that nonlinear eigenvalue problems can also be used for solving tasks arising in data science.

Note that eigenfunctions of linear operators have been used for decomposing and filtering data like audio {signals} for decades.
Their most prominent occurrence is the Fourier transform which decomposes signals into trigonometric functions, in other words eigenfunctions of the Laplacian operator {on the unit cube}.
While this is the most popular technique for audio processing, for the filtering of images the Fourier transform is only of limited use since it has a hard time resolving the discontinuities {naturally arising} as edges in images.
Instead, nonlinear operators like the $1$-Laplacian have turned out to be suitable tools for defining nonlinear spectral decompositions and filtering of data \citep{gilboa2018nonlinear,
gilboa2013spectral,
gilboa2014total,
benning2017nonlinear,
burger2016spectral,
fumero2020nonlinear,
burger2015spectral,
gilboa2016nonlinear}).
In these works the total variation flow \citep{andreu2002some} $\partial_t u(t) \in -\partial\tv(u(t))$ with $u(0)=f$ was utilized to define a nonlinear spectral representation of $f$ as $f = \int_0^T \phi(t)\d t$,
where the functions $\phi(t)$ for $t>0$ are computed from derivatives of the solution $u(t)$ and are (at least close to) eigenfunctions of the $1$-Laplacian operator $\Delta_1 u = \div(\nabla u/|\nabla u|)$, see \citet{gilboa2014total}.
\citet{bungert2019nonlinear} investigated these decompositions theoretically and proved conditions for the functions $\phi(t)$ to be nonlinear eigenfunctions.
The representation above can be used for multiple image processing tasks.
For example, in \cref{fig:santa} we show an application to face fusion \citep{benning2017nonlinear}, where the spectral representations of two images are combined.

Besides spectral decompositions, also the computation of ground states, i.e., eigenfunctions with minimal eigenvalue has important applications.
A prototypical example in data science is graph clustering, where eigenfunctions of the graph $p$-Laplacian operators \citep{elmoataz2015p} are used for partitioning graphs.
As observed by \citet{hein2010inverse,
buhler2009spectral,
bungert2019computing},
eigenfunctions of the nonlinear $1$-Laplacian are more suitable for clustering than Laplacian eigenfunctions since they allow for sharp discontinuities, see \cref{fig:graph_clustering_primer}.

\begin{figure}[t!]
    \def\Width{0.24\textwidth}
    \centering
    {\includegraphics[width=\Width]{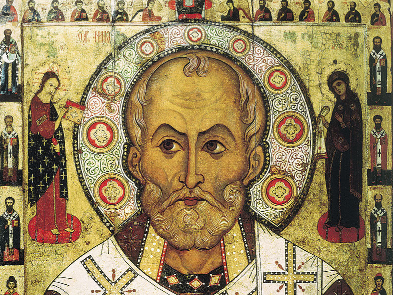}}%
    \hfill%
    {\includegraphics[width=\Width]{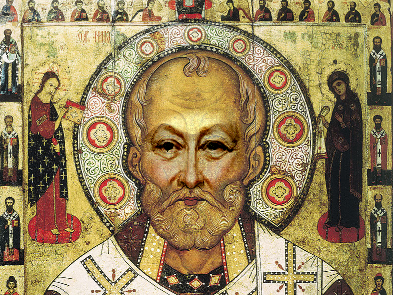}}%
    \hfill%
    {\includegraphics[width=\Width]{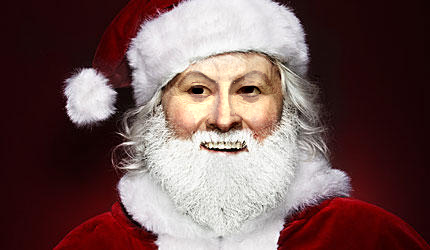}}%
    \hfill%
    {\includegraphics[width=\Width]{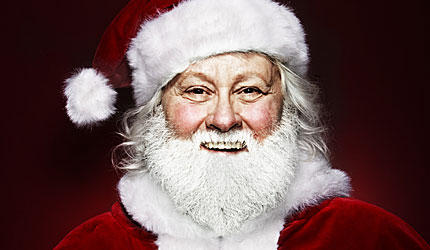}}%
    \caption{Nonlinear spectral face fusion using method by \citet{benning2017nonlinear}. Left to right: Saint Nicholas, Fusion ``Saint Nichoclaus'', Fusion ``Santolas'', Santa Claus}
    \label{fig:santa}
\end{figure}

\begin{figure}[t!]
    \def\Width{0.32\textwidth}
    \centering
    \begin{subfigure}{\Width}
    {\includegraphics[angle=0,width=\textwidth,trim=3cm 4cm 2.5cm 3.5cm,clip]{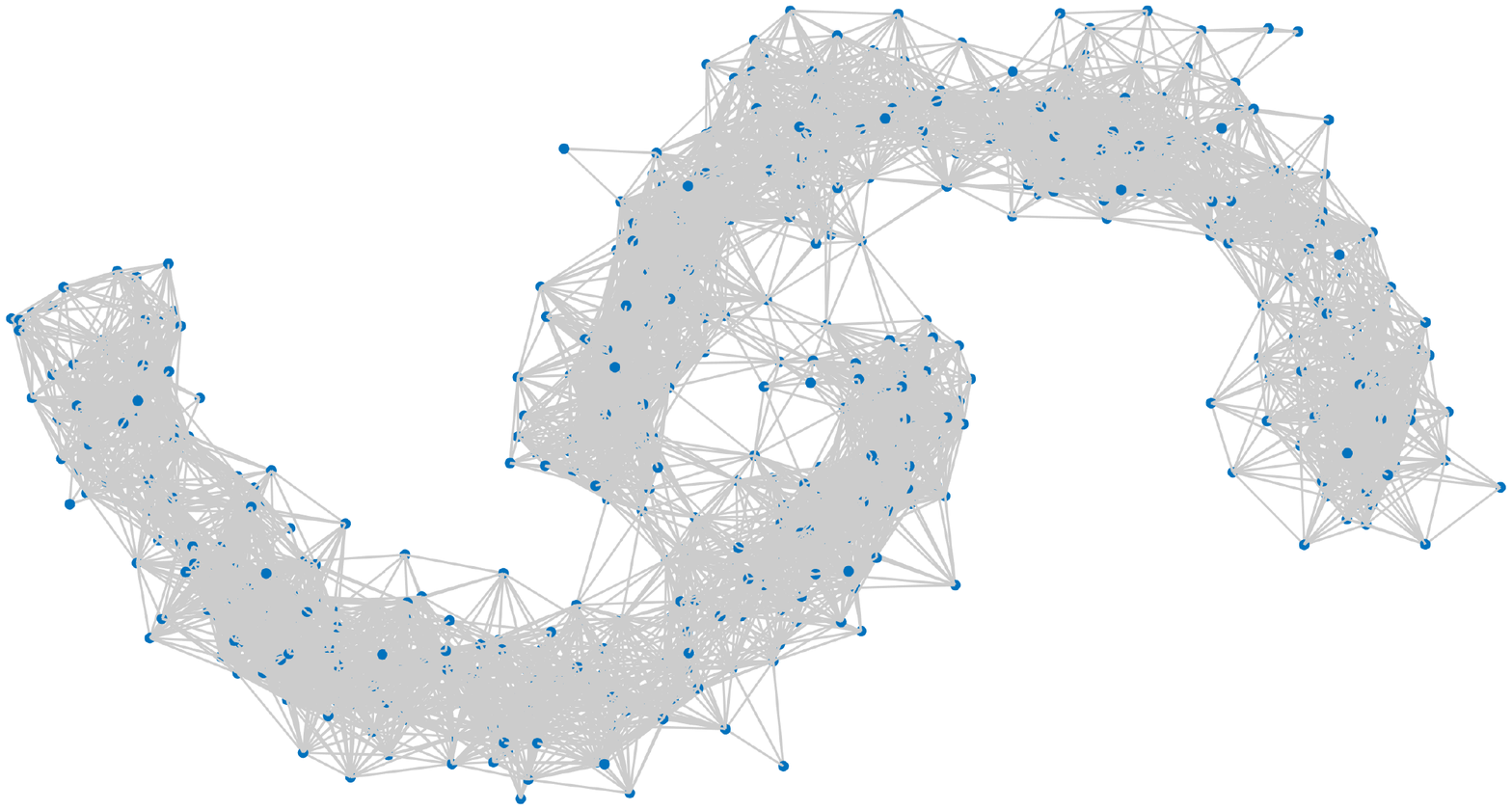}}%
    \caption{Two-moon graph}
    \end{subfigure}
    \hfill%
    \begin{subfigure}{\Width}
    {\includegraphics[angle=0,width=\textwidth,trim=3cm 4cm 2.5cm 3.5cm,clip]{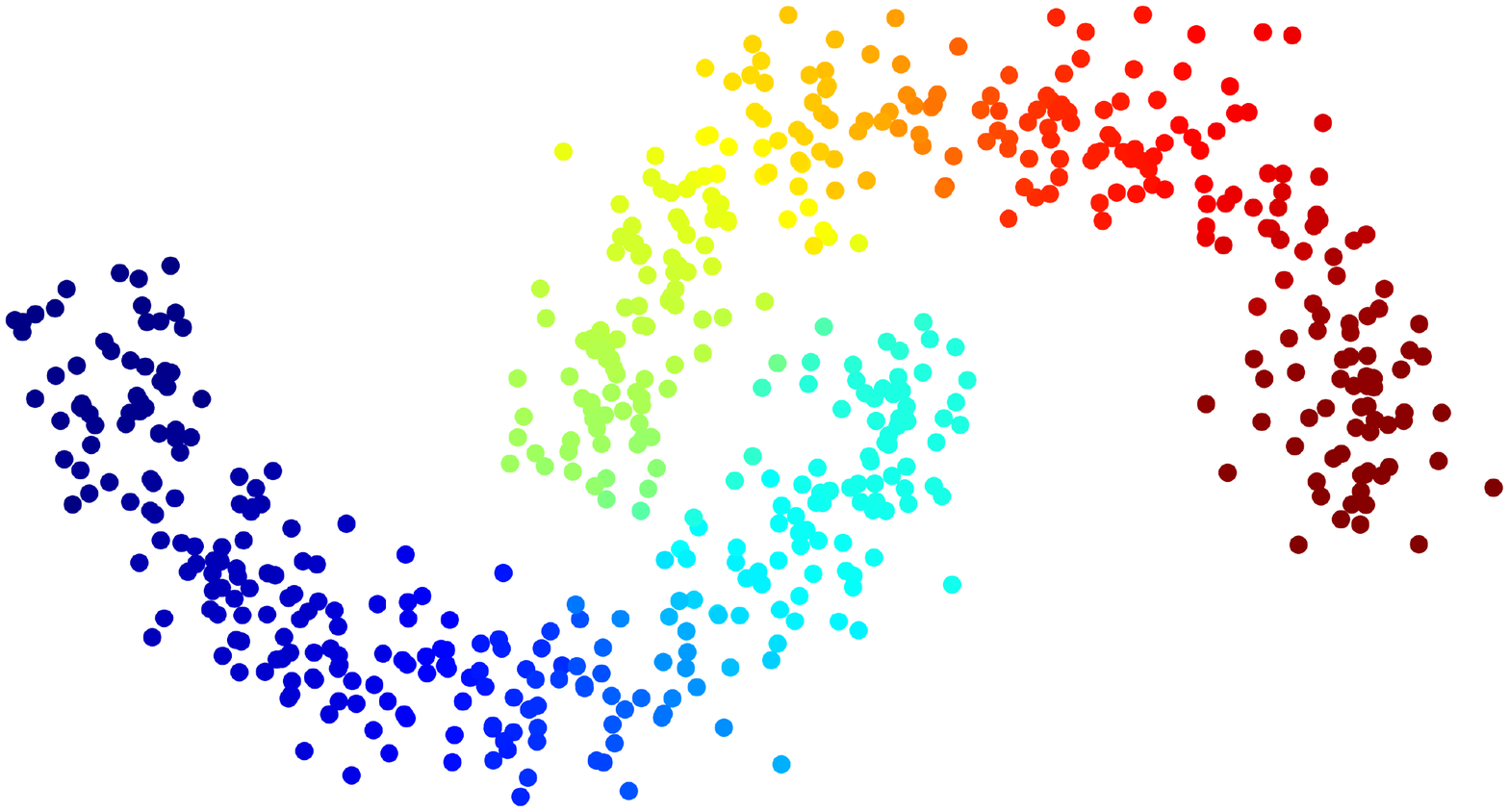}}%
    \caption{$2$-Laplacian eigenfunction}
    \end{subfigure}
    \begin{subfigure}{\Width}
    \hfill%
    {\includegraphics[angle=0,width=\textwidth,trim=3cm 4cm 2.5cm 3.5cm,clip]{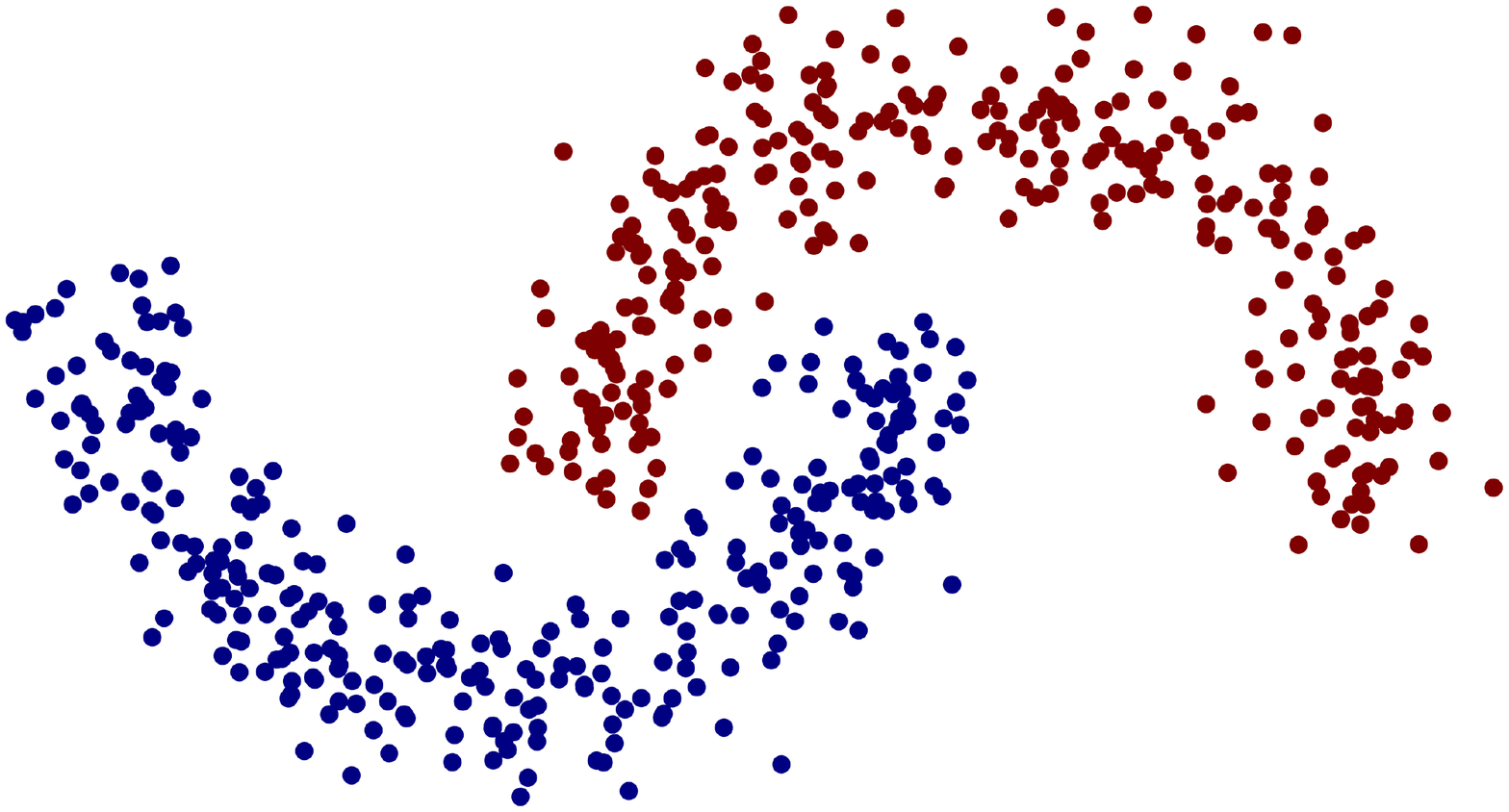}}%
    \caption{$1$-Laplacian eigenfunction}
    \end{subfigure}
    \caption{Spectral graph clustering}
    \label{fig:graph_clustering_primer}
\end{figure}

In this chapter we focus on the following class of variational nonlinear eigenvalue problems on a Banach space $\X$, where the nonlinear operators are subdifferentials.
For a convex function $f:\X\to(-\infty,\infty]$ its subdifferential is defined as
\begin{align}\label{eq:subdiff}
\partial f(u) := \{\sg\in\X^*\st f(u) + \langle\sg,v-u\rangle\leq f(v),\;\forall v\in\X\}.
\end{align}
Let us also define the so-called duality mapping of $\X$ with respect to a number $p\in[1,\infty)$ as $\Phi_\X^p = \partial\frac{1}{p}\norm{\cdot}^p$ (see \cref{sec:preliminaries} for details).

\begin{definition}[Nonlinear eigenvectors]\label{def:eigenvectors}
Let $\func:\X\to(-\infty,\infty]$ be proper and convex, and $p\in[1,\infty)$.
A tuple $(u,\hat{u})\in\X\times\argmin\func$ is called $p$-eigenvector of $\func$ with eigenvalue $\lambda\in\R$ if 
\begin{align}\label{eq:gen_ev_prob}
    \lambda\Phi_\X^p(\hat{u}-u) + \partial\func(u) \ni 0.
\end{align}
{If $\argmin\func$ is a singleton (in particular if $\hat u =0$), we refer simply to $u$ as the $p$-eigenvector.}
\end{definition}
Note that one could replace the duality map with the subdifferential of a different functional, however, for most nonlinear eigenvalue problems which are interesting in applications this is not necessary.
The reason why we define eigenvectors as tuples instead of single vectors will become clear in a second.
Of particular interest in many applications ranging from physics to data science are eigenvectors with minimal eigenvalue which are called \emph{ground states} and defined as follows:
\begin{definition}[Ground states]\label{def:ground_states}
Let $\func:\X\to(-\infty,\infty]$ be proper and $p\in[1,\infty)$.
A tuple $(u_*,\hat{u}_*)\in\X\times\argmin\func$ is called $p$-ground state of $\func$ if 
\begin{align}\label{eq:ground_states}
    (u_*,\hat{u}_*)\in\argmax_{\hat{u}\in\argmin\func}\argmin_{u\in\X}\frac{\func(u)-\func(\hat{u})}{\norm{u-\hat{u}}^p}.
\end{align}
If $\argmin \func$ is a singleton we refer to $u_*$ as the $p$-ground state, and $\argmax$ is not needed.
\end{definition}
Note that throughout the chapter $\argmin$ and $\argmax$ are subsets of $\X$.
\cref{prop:ground_states} below states that ground states are eigenvectors in the sense of \cref{def:eigenvectors} with minimal eigenvalue \citep{benning2012ground}.
Defining eigenvectors and ground states as tuples prevents minimizers of $\func$ from being trivial ground states.
We call the quotient appearing in \labelcref{eq:ground_states} Rayleigh quotient since it is a generalization of the standard Rayleigh quotient ${\langle u,Au\rangle}/{\norm{u}^2}$ from linear algebra.

While \cref{def:ground_states} and \labelcref{eq:gen_ev_prob} are quite abstract, they cover most interesting nonlinear eigenvalue problems from different fields.
To give an example from physics, for the energy $\func(u)=\int_{\R^d}\frac{1}{2}\abs{\nabla u}^2-\frac{1}{p+2}\abs{u}^{p+2}\d x$ on $\X=L^2(\R^d)$ the eigenvalue problem \labelcref{eq:gen_ev_prob} is characterized by $\hat{u}_*=0$ and $u_*$ solving the nonlinear Schrödinger equation \citep{weinstein1982nonlinear} $\lambda u = -\Delta u - \abs{u}^{p}u$.

In this chapter we explain how nonlinear eigenvectors and ground states can be computed.
Note that these problems are hard to solve for multiple reasons. 
Firstly, the eigenvalues in \labelcref{eq:gen_ev_prob} are unknown, secondly, the saddle point problem \labelcref{eq:ground_states} is neither convex in $u$ nor concave in $\hat{u}$, and lastly the functional $\func$ and the Banach space norm $\norm{\cdot}$ are non-smooth in interesting applications.
However, thanks to the variational structure of \cref{def:eigenvectors,def:ground_states} it turns out that gradient flows and their time discretizations are a useful tool.
 
Gradient flows are ubiquitous in mathematics and the natural sciences and in a Hilbert space setting \citep{brezis1973ope,aubin2012differential} they take the form
\begin{align}\label{eq:hilbert_gradflow}
    u'(t) + \partial \func(u(t)) \ni 0,
\end{align}
In the context of the nonlinear eigenvalue problem \labelcref{eq:gen_ev_prob}, we have to {work rather with} gradient flows in a Banach space.
To do this, one observes that the subdifferential \labelcref{eq:subdiff} naturally lives in the dual space and hence one has to ``push'' the time derivative to the dual space. 
This can be achieved by the duality mapping $\Phi_\X^p = \partial\frac{1}{p}\norm{\cdot}^p$ and leads to the equation
\begin{align}\label{eq:p-flow}
    \Phi_\X^p(u'(t)) + \partial\func(u(t)) \ni 0,
\end{align}
which is referred to as $p$-gradient flow or short $p$-flow.
More generally, one could replace $\Phi_\X^p(u')$ by a functional of the form $\partial_{u'}\mathcal{D}(u,u')$ \citep{mielke2016evolutionary}.
The equation \labelcref{eq:p-flow} is a special case of a metric gradient flow, for which a comprehensive theory is laid out in the monograph by \citet{ambrosio2008gradient}.

The relation between the gradient flow \labelcref{eq:p-flow} and the nonlinear eigenvalue problem \labelcref{eq:gen_ev_prob} is not completely obvious although they share structural similarities.
The simplest example for relations between gradient flows and ground states is the linear heat equation $u'(t)=\Delta u(t)$, posed on a bounded domain with Dirichlet boundary conditions. 
Expanding the solution in terms of an eigenbasis of the Laplacian operator one can show that if the initial datum satisfies $u(0)\geq 0$, then the normalized heat flow converges to an Laplacian eigenfunction:
\begin{align*}
    \lim_{t\to\infty}\norm{\frac{u(t)}{\norm{u(t)}_{L^2}} - w}_{L^2} = 0,
    \quad
    \text{where}
    \quad
    \lambda w = -\Delta w.
\end{align*}
The limit $w$ is referred to as \emph{asymptotic profile} or \emph{extinction profile} and describes the shape of the solution just before it converges (or extincts).
The existence of asymptotic profiles which are eigenvectors of the spatial differential operator has been proved for numerous (non)linear PDEs, including the total variation flow and $p$-Laplacian evolution equations \citep{bonforte2012total,andreu2002some,andreu2004parabolic,ghidaglia1991exact,kamin1988fundamental,portilheiro2013degenerate,vazquez2018asymptotic}, porous medium \citep{stan2018porous,vazquez2007porous} and fast diffusion equations \citep{bonforte2012behaviour,bonforte2020sharp}.

These results were generalized to the abstract equations \labelcref{eq:hilbert_gradflow} and \labelcref{eq:p-flow} by \citet{bungert2019nonlinear,bungert2019asymptotic,Bungert2020,hynd2017approximation,varvaruca2004exact} where asymptotic profiles where shown to satisfy \labelcref{eq:gen_ev_prob} or its Hilbertian version $\lambda u\in\partial\func(u)$.
A different branch of research \citep{feld2019rayleigh,aujol2018theoretical,cohen2018energy,nossek2018flows} studied Rayleigh quotient minimizing flows arising as gradient flows of the Rayleigh quotient \revise{$R(u) = {\func(u)}/{\norm{u}^p}$ for $u\neq 0$.}
The flows work well in practice and solve the same nonlinear eigenvalue problems at convergence, however, their rigorous mathematical analysis is more challenging. 
A comprehensive overview of these methods is given in the chapter by \citet{gilboa2020iterative}.
From a numerical perspective two major questions arise: 
\begin{itemize}
    \item Do also the time discretizations of all these flows, in particular of \labelcref{eq:p-flow}, converge to solutions of the eigenvalue problems \labelcref{eq:gen_ev_prob}?
    \item Can the eigenvalue problem~\labelcref{eq:gen_ev_prob} be approximated through discretization?
\end{itemize}
The first question was partially answered by \citet{Bungert2020,bungert2020nonlinear}, where it was observed that a time-discretization of normalized gradient flows yields a power method for the proximal operator of the energy functional, whose convergence was proved for absolutely one-homogeneous functionals on Hilbert spaces.
The second question was addressed for a special case by \citet{roith2020continuum}, where it was proved that $\Gamma$-convergence of a sequence of $L^\infty$-type functionals on weighted graphs implies convergence of their ground states, i.e., eigenvectors with minimal eigenvalue.

In this chapter we provide a comprehensive theory for the two questions  raised above. 
In \cref{sec:rayleigh} we first show that both gradient flows of general convex functionals and their implicit time discretization possess the astonishing property that they decrease the nonlinear Rayleigh quotient in \labelcref{eq:ground_states}.
\cref{sec:flows} then reviews several flows from the literature, which were proposed to solve nonlinear eigenvalue problems associated to homogeneous functionals. 
In particular, we show that all these flows are equivalent to a renormalization of the gradient flow.
In \cref{sec:power_method} we then thoroughly analyze a nonlinear power method based on the proximal operator, naturally arising as time-discretization of the flows from \cref{sec:flows}.
We prove convergence of the power method to a nonlinear eigenvector, show angular convergence, and prove that the power method preserves positivity.
In \cref{sec:gamma-cvgc} we then answer the second question by proving that $\Gamma$-convergence of general convex functionals implies convergence of the corresponding ground states.
We illustrate our theoretical findings in \cref{sec:applications} where we show numerical results of proximal power iterations, study their numerical convergence, and visualize the convergence of eigenvectors under $\Gamma$-convergence. 
Our numerical experiments are carried out both on regular grids and on general weighted graphs.

The appendix contains some of our proofs and provides further details on the phenomena of exact reconstruction and finite extinction which can occur in the evaluation of proximal operators and have to be taken into account for well-defined proximal power methods.

\section{Convex Analysis and Non\-lin\-ear Eigen\-val\-ue Problems}
\label{sec:preliminaries}

If $(\X,\norm{\cdot})$ is a Banach space we denote the dual space of $\X$ by $(\X^*,\norm{\cdot}_*)$ where 
\begin{align}
    \norm{\sg}_*&=\sup_{\substack{u\in\X\\\norm{u}=1}}\left\langle \sg,u\right\rangle,\quad\sg\in\X^*
\end{align}
is the dual norm.
\revise{For $p\in[1,\infty)$} the so-called $p$-proximal operator of a convex functional $\func:\X\to(-\infty,\infty]$ is defined as \revise{the set-valued function}
\begin{align}\label{eq:p-prox}
    \pprox_{\tau\func}(f) = \argmin_{u\in\X} \frac{1}{p}\norm{u-f}^p + \tau\func(u),\quad \tau \geq 0,
\end{align}
where the minimization problem has a solution under suitable conditions on the functional~$\func$.
\revise{For $\X=\H$ being a Hilbert space and $\func$ being convex, the $p$-proximal coincides with the standard proximal operator which is single-valued.}
The $p$-proximal operator has strong relations with the duality map of $\X$, given by
\begin{align}
    \Phi^p_\X(u):=
    \begin{cases}
    \{\sg \in \X^* \st \langle\sg,u\rangle=\norm{\sg}_*\norm{u},\,\norm{\sg}_*=\norm{u}^{p-1}\},\quad&p>1,\\
    \{\sg \in \X^* \st \langle\sg,u\rangle=\norm{u},\, \norm{\sg}_*\leq 1\},\quad&p=1. \\
    \end{cases}
\end{align}
Note that the two definitions differ only in the case $u=0$ where the duality map coincides with $\{0\}$ for $p>1$ and with $\{\sg\in\X^*\st\norm{\sg}\leq 1\}$ for $p=1$.
The following proposition collects some important properties of the duality map and the $p$-proximal operator.
\revise{This operator is in fact a special case of a larger class of proximity operators on Banach spaces studied by \citet{penot1998proximal,combettes2013moreau,wexler1973prox}.}
\begin{prop}[Properties of duality map and $p$-proximal]\label{prop:duality_proximal}
It holds that
\begin{itemize}
    \item $\Phi^p_\X(u)\neq \emptyset$ for all $u\in\X$,
    \item $\Phi^p_\X$ is a $p-1$-homogeneous map in the sense that
    \begin{align}
        \Phi^p_\X(cu) = c|c|^{p-2}\Phi^p_\X(u), \qquad\forall u\in\X,\; c\in\R,
    \end{align}
    \item $\Phi^p_\X = \partial\left(\frac{1}{p}\norm{\cdot}^p\right)$, where $\partial$ denotes the subdifferential,
    \item $u\in\pprox_{\tau\func}(f)$ if and only if
    \begin{align}\label{eq:OC}
        \exists\; \eta\in\Phi^p_\X(u-f),\;\sg\in\partial \func(u)\st 0=\eta+\tau \sg.
    \end{align}
\end{itemize}
\end{prop}
Many statements in this chapter concern ground states of a functional $\func$, see \cref{def:ground_states}.
As already mentioned in the introduction, ground states of convex functionals are solutions to a doubly nonlinear eigenvalue problem involving the duality map and the subdifferential of~$\func$, as the following proposition states, the proof of which can be found in the appendix.
\begin{prop}\label{prop:ground_states}
Let $\func:\X\to(-\infty,\infty]$ be convex and proper.
Let $(u_*,\hat{u}_*)\in\X\times\argmin\func$ be a $p$-ground state and define $\lambda_p:=p\frac{\func(u_*)-\func(\hat{u}_*)}{\norm{u_*-\hat{u}_*}^p}$.
Then it holds
\begin{align}
    \lambda_p\Phi_\X^p(\hat{u}_*-u_*) + \partial\func(u_*) \ni 0.
\end{align}
\end{prop}
Note that our definition of ground states and the relation to the nonlinear eigenvalue problem \labelcref{eq:gen_ev_prob} is valid for general convex functionals and does not require any homogeneity, which is typically assumed for nonlinear eigenvalue problems.
Still, some of our results are only valid for \emph{absolutely $\alpha$-homogeneous} functionals for some degree $\alpha\in[1,\infty)$, which by definition satisfy
\begin{align}\label{eq:alpha-hom}
    \func(cu) &= |c|^\alpha\func(u),\quad \forall u \in \dom\func.
\end{align}
Important examples for such functionals are 
\begin{alignat*}{3}
    \func(u)&=\sup_{\substack{\norm{\phi}_\infty=1}}\int_\Omega u \, \mathrm{div} \phi \,\d x, \quad&&\text{the total variation,}&&\alpha=1,\\
    \func(u)&=\int_\Omega|\nabla u|^p\d x, &&\text{the $p$-Dirichlet energy,}\quad&&\alpha=p.
\end{alignat*}
For instance, for the $p$-Dirichlet energy with homogeneous Dirichlet boundary conditions the abstract eigenvalue problem \labelcref{eq:gen_ev_prob} becomes the $p$-Laplacian eigenvalue problem
\begin{align*}
    \lambda_p|u|^{p-2}u = -\Delta_p u.
\end{align*}
For convex $\alpha$-homogeneous functionals we define their nullspace as
\begin{align}\label{eq:nullspace}
    \calN(\func):=\left\lbrace u \in \X \st \func(u) = 0\right\rbrace.
\end{align}
Because of the convexity and the homogeneity this is indeed a vector space, and for lower semicontinuous functionals it is additionally closed.

In the following we state the assumptions which we use in different parts of this chapter. 
For a compact presentation we make the following fundamental assumption which we assume without further notice and reference.
\begin{ass*}
The functional $\func:\X\to(-\infty,\infty]$ is proper, convex, and lower semicontinuous.
\end{ass*}
For studying convergence of rescaled gradient flows and proximal power methods to nonlinear eigenvectors in \cref{sec:flows} and \cref{sec:power_method} we need the following assumption.
\begin{ass}\label{ass:compactness}
The sub-level sets of $\norm{\cdot}+\func(\cdot)$ are relatively compact, in the sense that every sequence $(u_n)\subset\X$ such that $\sup_{n\in\N}\norm{u_n}+\func(u_n)<\infty$ admits a subsequence (which we do not relabel) and an element $u\in\X$ such that $\norm{u_n-u}\to 0$ as $n\to\infty$.
\end{ass}
In \cref{sec:power_method} we pose an assumption concerning absolutely $\alpha$-homogeneous functionals which demands that their nullspace be trivial.
We will see that this is no restriction of generality and makes the presentation more concise.
\begin{ass}\label{ass:hom_nullspace}
The functional $\func:\X\to(-\infty,\infty]$ is absolutely $\alpha$-homogeneous for some $\alpha\in[1,\infty)$ and satisfies $\calN(\func)=\{0\}$.
\end{ass}

\section{Gradient Flows and Decrease of Rayleigh Quotients}
\label{sec:rayleigh}

This section is dedicated to proving that under general conditions the $p$-flow \labelcref{eq:p-flow} and its implicit Euler discretization in time, often referred to as \emph{minimizing movement scheme}, both admit the astonishing property that they do not only decrease the energy $\func$ but also the nonlinear Rayleigh quotient in \labelcref{eq:ground_states}.
This property was already observed for the continuous-time gradient flows of homogeneous functionals, e.g., by \citet{ghidaglia1991exact,bungert2019asymptotic}.
However, for general convex functionals and for the time-discrete flows this statement appears to be novel.
Before we turn to the proofs we recap some basic facts about the $p$-flow and the associated minimizing movement scheme.

\subsection{The \texorpdfstring{$p$}{p}-Gradient Flow and Minimizing Movements}

The equation for a large part of this chapter is the $p$-gradient flow \labelcref{eq:p-flow} of a functional $\func$ on a Banach space $\X$ with initial datum $f\in\X$, which we repeat here for convenience:
\begin{align*}
    \Phi_\X^p(u'(t)) + \partial\func(u(t)) \ni 0,\qquad u(0)=f.
\end{align*}
It is a  doubly-nonlinear generalization of Hilbertian gradient flows~\labelcref{eq:hilbert_gradflow}, which are recovered by choosing $\X$ as a Hilbert space and $p=2$.
The right concept of solutions for the $p$-flow is the one of $p$-curves of maximal slope:
\begin{definition}
Let $1<p<\infty$. A $p$-curve of maximal slope for $\func$ is an absolutely continuous map $u:[0,\infty)\to\X$ that satisfies
\begin{align}\label{ineq:maxim_slope}
    \frac{\d}{\d t}\func(u(t)) \leq -\frac{1}{p}\abs{u'}^p(t) - \frac{1}{q}\abs{\partial\func}^q(u(t)),\quad a.e.\, t>0,
\end{align}
where $\abs{u'}(t):=\lim_{h\to 0}\frac{\norm{u(t+h)-u(t)}}{|h|}$ denotes the \emph{metric derivative} of $u$ in $t>0$ and $\abs{\partial\func}(u):=\inf\left\lbrace\norm{\sg}_*\st\sg\in\partial\func(u)\right\rbrace$ denotes the \emph{local slope} of $\func$ in $u\in\X$.
\end{definition}
Note that the opposite inequality to \labelcref{ineq:maxim_slope} is always satisfied thanks to the Cauchy-Schwarz and Young inequalities. 
This enables one to show that $p$-curves of maximal slope are in one-to-one correspondence to almost everywhere differentiable solutions of \labelcref{eq:p-flow}, see \citet{hynd2017approximation}.
Existence of solutions is discussed, e.g., by \citet{ambrosio2008gradient,santambrogio2017euclidean} (see also \citet{stefanelli2021new} for a new approach) but in this chapter we rather focus on qualitative properties of the flows.

For proving existence, most approaches use the concept of minimizing movements \citep{de1993new,ambrosio1995minimizing,ambrosio2008gradient}, which is an implicit Euler discretization in time of the $p$-flow \labelcref{eq:p-flow}.
Besides being a theoretical tool, it also serves as a numerical scheme for approximating solutions. 
The scheme arises by discretizing the values of $u(t)$ one discrete time steps $k\tau$ where $k\in\N_0$ and $\tau>0$ is a time step size.
Hence, the implicit discretization of \labelcref{eq:p-flow} takes the form
\begin{subequations}\label{eq:minimizing movements}
\begin{align}
    &\Phi_\X^p\left(\frac{u^{k+1}-u^k}{\tau}\right) + \partial\func(u^{k+1}) \ni 0, \\
    \iff
    &u^{k+1} \in \argmin_{u^\in\X} \frac{1}{p}\norm{u-u^k}^p + \tau^{p-1}\func(u), \\
    \iff 
    &u^{k+1} \in \pprox_{\tau^{p-1}\func}(u^k),
\end{align}
\end{subequations}
where we set $u^0:=f$.
As the time step size $\tau$ tends to zero, under generic conditions a time interpolation of the sequence generated by~\labelcref{eq:minimizing movements} converges to a curve of maximal slope which is a solution of \labelcref{eq:p-flow} (see, e.g., \citet{ambrosio2008gradient,santambrogio2017euclidean}).
This yields both existence of solutions and a numerical method to compute them.

\subsection{Decrease of the Rayleigh Quotients}
First, we show that the $p$-flow \labelcref{eq:p-flow} of a convex functional decreases the generalized Rayleigh quotient from~\labelcref{eq:ground_states}.

\begin{prop}\label{prop:decay_rayleigh_p-flow}
Let $u:[0,\infty)\to\X$ be an absolutely continuous solution of~\labelcref{eq:p-flow} and let $\func$ be convex.
Then it holds for all $\hat{u}\in\argmin\func$
\begin{align}
    \frac{\d}{\d t}\frac{\func(u(t))-\func(\hat{u})}{\norm{u(t)-\hat{u}}} 
    \leq 0,\quad\forall 0<t<\infty.
\end{align}
If $\func$ is absolutely $\alpha$-homogeneous for $\alpha\in[1,\infty)$, it even holds
\begin{align}\label{ineq:decrease_rayleigh_p}
    \frac{\d}{\d t}\frac{\alpha\func(u(t))}{\norm{u(t)-\hat{u}}^\alpha} 
    \leq 0,\quad\forall 0<t<\infty.
\end{align}
\end{prop}
\begin{proof}
Using the triangle inequality it can be easily seen \citep{hynd2017approximation} that
\begin{align*}
    \left|\frac{\d}{\d t}\norm{u(t)-\hat{u}}\right| \leq \norm{u'(t)}.
\end{align*}
Furthermore, one has by properties of $p$-flows \citep{hynd2017approximation}
\begin{align*}
    \frac{\d}{\d t}\func(u(t)) = -\norm{u'(t)}^p.
\end{align*}
Furthermore, for all $\sg\in\partial\func(u(t))$ one has
\begin{align*}
    \func(u(t))-\func(\hat{u}) \leq \norm{\sg}_*\norm{u(t)-\hat{u}}.
\end{align*}
By definition of the $p$-flow, for every $t>0$ there exists $\eta(t)\in\Phi_\X^p(u'(t))$ and $\sg(t)\in\partial\func(u(t))$ such that $\eta(t)+\sg(t)=0$.
Putting these four statements together, one obtains
\begin{align*}
    \frac{\d}{\d t}\frac{\func(u(t))-\func(\hat{u})}{\norm{u(t)-\hat{u}}} 
    &= - \frac{\left(\func(u(t))-\func(\hat{u})\right)\frac{\d}{\d t}\norm{u(t)-\hat{u}}}{\norm{u(t)-\hat{u}}^2} + \frac{\frac{\d}{\d t}\func(u(t))}{\norm{u(t)-\hat{u}}}  \\
    &\leq \frac{1}{\norm{u(t)-\hat{u}}} \left[ \norm{\sg(t)}_*\norm{u'(t)} - \norm{u'(t)}^p \right] \\
    &= \frac{1}{\norm{u(t)-\hat{u}}} \left[ \norm{\eta(t)}_*\norm{u'(t)} - \norm{u'(t)}^p \right] \\
    &=\frac{1}{\norm{u(t)-\hat{u}}} \left[ \norm{u'(t)}^p - \norm{u'(t)}^p \right] =0,
\end{align*}
where we used the properties of the duality mapping $\Phi_\X^p$.

In the case that $\func$ is absolutely $\alpha$-homogeneous we can use \citep{bungert2019asymptotic} 
\begin{align*}
    \alpha\func(u(t)) \leq \norm{\sg}_*\norm{u(t)-\hat{u}},\quad\forall\sg\in\partial\func(u(t)),
\end{align*}
to compute
\begin{align*}
    \frac{\d}{\d t}\frac{\alpha\func(u(t))}{\norm{u(t)-\hat{u}}^\alpha} 
    &= - \frac{\alpha\func(u(t))\frac{\d}{\d t}\norm{u(t)-\hat{u}}^\alpha}{\norm{u(t)-\hat{u}}^{2\alpha}} + \frac{\alpha\frac{\d}{\d t}\func(u(t))}{\norm{u(t)-\hat{u}}^{\alpha}}  \\
    &= \frac{\alpha\func(u(t))\alpha \norm{u(t)-\hat{u}}^{\alpha-1}\norm{u'(t)}}{\norm{u(t)-\hat{u}}^{2\alpha}} - \frac{\alpha\norm{u'(t)}^p}{\norm{u(t)-\hat{u}}^{\alpha}}  \\
    &\leq \frac{\alpha}{\norm{u(t)-\hat{u}}^\alpha} \left[ \frac{\alpha\func(u(t))\norm{u'(t)}}{\norm{u(t)-\hat{u}}} - \norm{u'(t)}^p \right] \\
    &\leq  \frac{\alpha}{\norm{u(t)-\hat{u}}^\alpha} \left[ {\norm{\sg(t)}_*\norm{u'(t)}} - \norm{u'(t)}^p \right] \leq 0,
\end{align*}
where the last inequality follows again from $-\sg(t)=\eta(t)\in\Phi_\X^p(u'(t))$.
\end{proof}
Next, we show that also the sequence $(u^{k})_{k\in\N}\subset\X$ generated by the minimizing movement scheme \labelcref{eq:minimizing movements} decreases the Rayleigh quotient in the precise same way, which is an important consistency property for this scheme.
\begin{prop}\label{prop:decrease_rayleigh}
Let $\tau>0$, $u^{k+1}\in\pprox_{\tau\func}(u^k)$, and $\hat{u}\in\argmin \func$.
Then it holds that 
\begin{align}\label{ineq:decrease_rayleigh}
    \frac{\func(u^{k+1})-\func(\hat{u})}{\norm{u^{k+1}-\hat{u}}}\leq\frac{\func(u^k)-\func(\hat{u})}{\norm{u^k-\hat{u}}}
\end{align}
with the convention $\frac00:=0$. 
If $\func$ is absolutely $\alpha$-homogeneous for $\alpha\in[1,\infty)$, it even holds
\begin{align}\label{ineq:decrease_rayleigh_p_discrete}
    \frac{\alpha\func(u^{k+1})}{\norm{u^{k+1}-\hat{u}}^\alpha}\leq\frac{\alpha\func(u^k)}{\norm{u^k-\hat{u}}^\alpha},\quad \forall \hat{u} \in \calN(J).
\end{align}
\end{prop}
\begin{proof}
Let us fix $\hat{u}\in\argmin\func$.
From the minimality of $u^{k+1}$ we infer
\begin{align}\label{ineq:optimality}
    \frac{1}{p}\norm{u^{k+1}-u^k}^p+\tau\func(u^{k+1}) \leq \frac{1}{p}\norm{v-u^k}^p+\tau\func(v),\quad \forall v\in\X.
\end{align}
We now distinguish two cases, starting with the easier one.
\\
\textbf{Case 1:}
Here, we assume that $\norm{u^{k+1}-\hat{u}}\geq\norm{u^k-\hat{u}}$.
Choosing $v=u^k$ in \labelcref{ineq:optimality} one obtains $\func(u^{k+1})\leq\func(u^k)$.
We can combine these two estimates to
\begin{align*}
    \frac{\func(u^{k+1})-\func(\hat{u})}{\norm{u^{k+1}-\hat{u}}} \leq \frac{\func(u^{k})-\func(\hat{u})}{\norm{u^{k}-\hat{u}}}.
\end{align*}
\\
\textbf{Case 2:}
Now, we assume that $\norm{u^{k+1}-\hat{u}}\leq\norm{u^k-\hat{u}}$.
Choosing 
$$v:=\frac{\norm{u^{k+1}-\hat{u}}}{\norm{u^k-\hat{u}}}(u^k-\hat{u})+\hat{u},$$ 
which satisfies 
\begin{align*}
    \norm{v-u^k}&=\norm{\frac{\norm{u^{k+1}-\hat{u}}}{\norm{u^k-\hat{u}}}(u^k-\hat{u})-(u^k-\hat{u})} 
    = \left|\frac{\norm{u^{k+1}-\hat{u}}}{\norm{u^k-\hat{u}}}-1\right|\norm{u^k-\hat{u}} \\
    &= \left| \norm{u^{k+1}-\hat{u}} - \norm{u^k-\hat{u}} \right| 
    \leq \norm{u^{k+1}-u^k},
\end{align*}
we obtain
\begin{align*}
    \func(u^{k+1}) &\leq  \func(v) = \func\left(\frac{\norm{u^{k+1}-\hat{u}}}{\norm{u^k-\hat{u}}}(u^k-\hat{u})+\hat{u}\right) \\
    &= \func\left(\frac{\norm{u^{k+1}-\hat{u}}}{\norm{u^k-\hat{u}}}u^k + \left(1-\frac{\norm{u^{k+1}-\hat{u}}}{\norm{u^k-\hat{u}}}\right)\hat{u}\right).
\end{align*}
Since $\norm{u^{k+1}-\hat{u}}\leq\norm{u^k-\hat{u}}$, we obtain using the convexity of $\func$
\begin{align*}
    \func(u^{k+1}) &\leq \frac{\norm{u^{k+1}-\hat{u}}}{\norm{u^k-\hat{u}}}\func(u^k) + \left(1-\frac{\norm{u^{k+1}-\hat{u}}}{\norm{u^k-\hat{u}}}\right) \func(\hat{u}).
\end{align*}
This inequality can be reordered to~\labelcref{ineq:decrease_rayleigh}.

If $\func$ is absolutely $\alpha$-homogeneous, we can use that $\func(u^{k+1}+\hat{u})=\func(u^{k+1})$ for all $\hat{u}\in\argmin\func=\calN(J)$ (see~\citet{bungert2019asymptotic}) and obtain
\begin{align*}
    \func(u^{k+1}) &\leq \left(\frac{\norm{u^{k+1}-\hat{u}}}{\norm{u^k-\hat{u}}}\right)^\alpha\func(u^k),
\end{align*}
which can be reordered to~\labelcref{ineq:decrease_rayleigh_p_discrete}.
\end{proof}

\section{Flows for Solving Nonlinear Eigenproblems}
\label{sec:flows}

In this section we review several flows from the literature whose large time behavior has been brought into correspondence with ground states according to \cref{def:ground_states} or---more generally---nonlinear eigenfunctions solving~\labelcref{eq:gen_ev_prob}. 
The relation between gradient flows and eigenvectors of general subdifferential operators was studied by \citet{bungert2019asymptotic,bungert2019computing,bungert2019nonlinear,varvaruca2004exact,hynd2017approximation}, see also \citet{Bungert2020}.
Rayleigh quotient minimizing flows were investigated by \citet{feld2019rayleigh,cohen2018energy,aujol2018theoretical,nossek2018flows}, see also \citet{gilboa2020iterative,gilboa2018nonlinear}.
In the following we review these flows and discover that they have strong relations and are connected through time reparametrizations and normalizations.

The first three approaches below are all based on the $p$-flow \labelcref{eq:p-flow} or gradient flows on Hilbert spaces \labelcref{eq:hilbert_gradflow}.
The fourth class of flows we study here are Rayleigh quotient minimizing flows on Hilbert spaces which are seemingly different from gradient flows.
However, we show that they all arise as time rescalings of normalized gradient flows.
Note that all approaches utilize absolutely $\alpha$-homogeneous functionals (see \labelcref{eq:alpha-hom}).

\subsection{Rescaled \texorpdfstring{$p$}{p}-Flows}

The first approach was investigated by \citet{hynd2017approximation} and considers the $p$-flow \labelcref{eq:p-flow} of absolutely $p$-homogeneous functionals on Banach spaces, i.e., the homogeneity of the functional coincides with~$p$.
The authors showed that in this case solutions of the $p$-flow decay exponentially and that their exponential rescalings converge to a limit, which was characterized to be zero or a nonlinear eigenfunction in the sense of \labelcref{eq:gen_ev_prob}.

To set the scene, let $\func:\X \to (-\infty,\infty]$ be absolutely $p$-homogeneous and define
\begin{align}
    \lambda_p := \inf_{u\neq 0} \frac{p\func(u)}{\norm{u}^p}.
\end{align}
The value $\lambda_p$ is called simple if any two minimizers are linearly dependent. 
Furthermore, from \cref{prop:ground_states} we see that $\lambda_p$ is the smallest eigenvalue of the eigenvalue problem~\labelcref{eq:gen_ev_prob}.

\begin{thm}[Theorem 1.3 by \citet{hynd2017approximation}]
Let $p>1$ and $u:[0,\infty)\to\X$ be a solution of~\labelcref{eq:p-flow} with $\func$ being absolutely $p$-homogeneous.
Assume that $\lambda_p>0$ and simple and denote $\mu_p:=\lambda_p^\frac{1}{p-1}$. 
Then it holds
\begin{itemize}
    \item The limit $w:=\lim_{t\to\infty}e^{\mu_p t}u(t)$ exists.
    \item If $w\neq 0$, then $w$ is a $p$-ground state and it holds $\lambda_p=\lim_{t\to\infty}\frac{p\func(u(t))}{\norm{u(t)}^p}$.
\end{itemize}
\end{thm}

\begin{example}
In this example we let $\X=L^p(\Omega)$ for $p>1$ and consider the functional $\func(u)=\int_\Omega|\nabla u|^p\d x$ if $u\in W^{1,p}_0(\Omega)$ and $\func(u)=\infty$ else.
If the initial datum of~\labelcref{eq:p-flow} is positive, i.e., $f(x)\geq 0$ for almost all $x\in\Omega$, then $w$ is non-zero and is a solution of the $p$-Laplacian eigenvalue problem $\lambda_p|u|^{p-2}u = -\Delta_p u$.
\end{example}

\subsection{Rescaled Gradient Flows}

The second approach we review here was investigated by \citet{bungert2019nonlinear,bungert2019asymptotic,Bungert2020} and deals with Hilbertian gradient flows of absolutely $\alpha$-homogeneous functionals for $\alpha\in[1,\infty)$, meaning that the homogeneity of the functional not necessarily coincides $p=2$, which is the appropriate parameter for the duality map in \labelcref{eq:p-flow} in the case of Hilbert space gradient flows.
It was shown that, depending the homogeneity $\alpha$, solutions decay either in finite time, exponentially, or algebraically (see also \citet{hauer2019kurdyka} for analogous statements for metric gradient flows) and that suitable rescalings converge. 
For all $\alpha\neq 2$ the limit was shown to be a non-zero solution of the eigenvalue problem $\lambda u \in \partial \func(u)$ which is the Hilbert space version of the eigenvalue problem \labelcref{eq:gen_ev_prob}.
For $\alpha=2$ which coincides with $p=2$ the limit can be zero, just as for the previous approach where the homogeneity of the functional coincides with the duality parameter~$p$.

We now adopt the following setting: Let $\X = \H$ be a Hilbert space and $\func:\H \to \R\cup\{\infty\}$ be absolutely $\alpha$-homogeneous.
Setting $p=2$ in \labelcref{eq:p-flow} yields the gradient flow
\begin{align}\label{eq:gradient_flow}
    u'(t) + \partial \func(u(t)) \ni 0, \qquad u(0) = f.
\end{align}
In order to characterize the decay of solution to this flow one consider the following initial value problem, depending on $\lambda>0$
\begin{align}\label{eq:ODE}
    a'(t) = -\lambda a(t)^{\alpha-1}, \qquad a(0) = 1,
\end{align}
which arises as special case of \labelcref{eq:gradient_flow} by setting $\H=\R$, $\func(x) = \frac{\lambda}{\alpha}|x|^\alpha$, and $f=1$.

Slightly more general as above we define the minimal eigenvalue as
\begin{align}\label{eq:eigenvalue_hilbert}
    \lambda_\alpha := \inf_{u\in\calN(\func)^\perp} \frac{\alpha\func(u)}{\norm{u}^\alpha} = \inf_{u\notin\calN(\func)}\frac{\alpha\func(u)}{\norm{u-\overline{u}}^\alpha},
\end{align}
where the \revise{unique} orthogonal projection of $u\in\H$ onto the nullspace $\calN(\func)$ is defined by 
\begin{align}
    \overline{u}:=\argmin_{v\in\calN(\func)}\norm{u-v}.
\end{align}
Since for absolutely homogeneous functionals it holds $\argmin\func=\calN(\func)$ and $\min\func=0$, problem \labelcref{eq:eigenvalue_hilbert} coincides precisely the eigenvalue of ground states in the sense of \cref{def:ground_states}.

As shown by \citet{bungert2019asymptotic} the dynamics of~\labelcref{eq:ODE} completely describe the asymptotic behavior of the much more general gradient flow~\labelcref{eq:gradient_flow}.
Furthermore, rescaling the gradient flow with a solution of~\labelcref{eq:ODE}, yields convergence to an eigenfunction.

\begin{thm}[Theorems 2.3, 2.4 by \citet{bungert2019asymptotic}]
Let $u:[0,\infty)\to\H$ be a solution of~\labelcref{eq:gradient_flow} with $\func$ being absolutely $\alpha$-homogeneous.
Let furthermore $\tex:=\inf\{t > 0 \st u(t) = \overline{f}\}$ denote the extinction time of~\labelcref{eq:gradient_flow} and
\begin{align*}
    \lambda:=
    \begin{cases}
    \frac{1}{(2-\alpha)\tex},\qquad &1\leq \alpha < 2, \\
    \lambda_\alpha\norm{f}^{\alpha-2},\qquad &2 \leq \alpha < \infty.
    \end{cases}
\end{align*}
Then there exists an increasing sequence $(t_k)_{k\in\N}$ with $t_k \nearrow \tex$, and an element $w\in\H$ with 
\begin{align*}
    \lim_{k\to\infty}\frac{u(t_k)-\overline{f}}{a_\lambda(t_k)} = w, \qquad
    \lambda w \in \partial\func(w),
\end{align*}
where $a_\lambda:[0,\infty)\to[0,1]$ denotes the solution of \labelcref{eq:ODE}.
Furthermore, it holds:
\begin{itemize}
    \item If $\alpha\neq 2$ then $w \neq 0$.
    \item If $\alpha=2$ and $w\neq 0$, then $w$ is a $2$-ground state and it holds $\lambda_2=\lim_{t\to\infty}\frac{2\func(u(t))}{\norm{u(t)-\overline{f}}^2}$.
\end{itemize}
\end{thm}

\begin{example}
Here we let $\H=L^2(\Omega)$ and define the absolutely one-homogeneous functional $\func(u)=\esssup_{x\in\Omega} |\nabla u(x)|$ if $u\in W^{1,\infty}_0(\Omega)$ and $\func(u)=\infty$ else.
If the initial datum of~\labelcref{eq:gradient_flow} is positive, i.e., $f(x)\geq 0$ for almost all $x\in\Omega$, then $w$ is non-zero and a solution of the eigenvalue problem $\lambda u\in \partial\func(u)$, which was shown by \citet{bungert2020structural} to equal a multiple of the Euclidean distance function $x\mapsto \dist(x,\partial\Omega)$ (see~\citet{roith2020continuum} for the analogous statement for geodesic distance functions).
\end{example}

\subsection{Normalized Gradient Flows}

The previous rescaling factors $e^{\mu_p t}$ and $1/a_\lambda(t)$, respectively, allow one to investigate whether the rescaled solutions converge to ground states, i.e., eigenvectors with minimal eigenvalue.
However, both rescalings depend on numbers which are not available in applications (the eigenvalue $\lambda_p$ and the extinction time $\tex$).
Therefore, one can study rescalings with the norm of the solution itself, i.e.,
\begin{align}\label{eq:normalized_gradflow}
\begin{cases}
    u'(t) + \partial\func(u(t)) \ni 0, \quad u(0) = f, \\
    w(t) = \frac{u(t)- \overline{f}}{\norm{u(t)-\overline{f}}},
\end{cases}
\end{align}
where $\func$ is as before. 
Since the proofs of the previous theorems largely depend on the statement of \cref{prop:decay_rayleigh_p-flow}, it does not matter too much which rescaling one chooses as long as one can make sure that the rescaled gradient flow remains uniformly bounded.

The following result was proven by \citet{varvaruca2004exact} for functionals that are locally subhomogeneous, however, for brevity we state it for homogeneous functionals.
It asserts that the rescalings from \labelcref{eq:normalized_gradflow} converge to eigenvectors, however, does not give conditions for this eigenvector being a ground state. 

\begin{thm}[Theorem 4.1 by \citet{varvaruca2004exact}]
Let $\alpha\geq 2$ and $u:[0,\infty)\to\H$ be a solution of~\labelcref{eq:gradient_flow} with $\func$ being absolutely $\alpha$-homogeneous. 
Then there exists an increasing sequence $(t_k)_{k\in\N}$ with $t_k \to \infty$, and an element $w\in\H$ such that 
\begin{align*}
    \lim_{k\to\infty} \frac{u(t_k)-\overline{f}}{\norm{u(t_k)-\overline{f}}^\alpha} = w, \qquad
    \lim_{k\to\infty} \frac{\alpha\func(u(t_k))}{\norm{u(t_k)-\overline{f}}^\alpha} = \lambda, \qquad
    \lambda w \in \partial \func(w).
\end{align*}
\end{thm}

\begin{rem}
Using the techniques from \citet{bungert2019asymptotic} and introducing extinction times, it is easy to prove that this theorem also holds true for all $\alpha\geq 1$. 
One only has to replace the sequences $t_k\to\infty$ with sequences $t_k\to\tex$.
\end{rem}

We conclude this section by an interesting observation.
One can show that time rescalings $w(t)$ in \labelcref{eq:normalized_gradflow} also solve a suitable flow, which is the statement of the following proposition.
\begin{prop}
Let $w(t)$ be given by~\labelcref{eq:normalized_gradflow} for $t\in[0,\tex)$ and let $v(t):=w\left(\norm{u(t)}^{2-\alpha}t\right)$.
Then $v$ solves
\begin{align}\label{eq:rescaled_flow}
    v' \in \alpha\func(v) v - \partial\func(v),\quad v(0)=\frac{f-\overline{f}}{\norm{f-\overline{f}}}.
\end{align}
\end{prop}
\begin{proof}
Since $t\mapsto u(t)$ is Lipschitz continuous and right differentiable for $t>0$ and $u(t)\neq 0$ for all $t\in[0,\tex)$, the same holds true for $t\mapsto\norm{u(t)}$ and $t\mapsto w(t)$.
Using the quotient rule and letting $\sg(t):=-u'(t)\in\partial\func(u(t))$, one gets
\begin{align*}
    w'(t) &= \frac{\norm{u(t)}u'(t) - \langle u'(t),u(t)\rangle \frac{u(t)}{\norm{u(t)}}}{\norm{u(t)}^2} 
    = \frac{\alpha\func(u(t))}{\norm{u(t)}^2} w(t) -\frac{\sg(t)}{\norm{u(t)}}  \\
    &= \norm{u(t)}^{\alpha-2}\alpha\func(w(t)) w(t) -\frac{\sg(t)}{\norm{u(t)}} 
    = \norm{u(t)}^{\alpha-2}\left[\alpha\func(w(t))w(t) - \frac{\sg(t)}{\norm{u(t)}^{\alpha-1}} \right].
\end{align*}
Now we use that $\partial\func$ is $(\alpha-1)$-homogeneous to obtain that $\sg(t)/\norm{u(t)}^{\alpha-1}\in\partial\func(w(t))$.
Hence, the time rescaling $v(t):=w(\norm{u(t)}^{2-\alpha}t)$ solves the flow \labelcref{eq:rescaled_flow}.
\end{proof}

\subsection{Rayleigh Quotient Minimizing Flows}

Flows similar to~\labelcref{eq:rescaled_flow} have been proposed in the literature a lot, however, without explicitly relating to gradient flows.
A comprehensive overview is given by \citet{gilboa2020iterative}.
For instance, \citet{feld2019rayleigh} studied a Rayleigh quotient minimizing flow which takes the form
\begin{align}\label{eq:FAGP-flow}
    \begin{cases}
        u(0) = f, \\
        u' = R(u) q - \sg,\quad q \in \partial H(u),\,\sg\in\partial \func(u).
    \end{cases}
\end{align}
The flow was derived in order to minimize the Rayleigh quotient
\begin{align}\label{eq:rayleigh_J_R}
    R(u) = \frac{\func(u)}{H(u)},
\end{align}
for absolutely one-homogeneous functionals $\func$ and $H$ defined on a Hilbert space $\H$ and extends previous results from \citet{aujol2018theoretical}.
In applications, often $H(u)=\norm{u}_\X$ where $\X$ is a Banach space larger than $\H$, e.g., $\X=L^1(\Omega)$ and $\H=L^2(\Omega)$.

The flow \labelcref{eq:FAGP-flow} has the property that $\frac{\d}{\d t}\frac{1}{2}\norm{u(t)}^2=0$, hence preserving the norm of the initial condition. 
Consequently, for the choice $H(u)=\norm{u}$ and $\norm{f}=1$ the flow \labelcref{eq:FAGP-flow} reduces to \labelcref{eq:rescaled_flow}, which is equivalent to the normalized gradient flow \labelcref{eq:normalized_gradflow}, as we showed in the previous section.

A similar evolution was initially studied by \citet{nossek2018flows} for the minimization of $\func(u)/\norm{u}$ for one-homogeneous $\func$.
For the minimization of the general Rayleigh quotient~\labelcref{eq:rayleigh_J_R} we can generalize it as follows
\begin{align}\label{eq:nossek-flow}
\begin{cases}
    v(0) = f, \\
    v' = r - \frac{\eta}{H_*(\eta)},\quad r \in \partial H(v),\;\eta\in\partial \func(v).
\end{cases}
\end{align}
Here $H_*$ is the dual seminorm (see, e.g., \citet{bungert2019nonlinear,Bungert2020} for properties) to the absolutely 1-ho\-mo\-ge\-neous functional $H$, defined as
\begin{align}\label{eq:dual_semi-norm}
    H_*(\sg) = \sup_{\substack{u\in\H\\H(u)=1}}\langle\sg,u\rangle,\quad\sg\in\calN(H)^\perp.
\end{align}
Interestingly, as we show in the following proposition, this flow is asymptotically equivalent to~\labelcref{eq:FAGP-flow}.
In particular, for $H(u)=\norm{u}$ the original method from \citet{nossek2018flows} is asymptotically equivalent to the normalized gradient flow \labelcref{eq:normalized_gradflow}.

\begin{prop}
Let $v$ be a solution of \labelcref{eq:nossek-flow} and let $\phi:[0,\infty) \to [0,\infty)$ solve the initial value problem
\begin{align}\label{eq:ivp}
    \phi'(t) = \frac{\func(v(\phi(t)))}{H(v(\phi(t)))},\quad \phi(0) = 0.
\end{align}
Then $u(t):=v(\phi(t))$ solves 
\begin{align}\label{eq:asymptotic-nossek-flow}
    u' = R(u)q - \frac{\func(u)}{H(u)H_*(\sg)}\sg, \quad q\in\partial H(u),\,\sg \in \partial \func(u).
\end{align}
\end{prop}
\begin{proof}
Defining $\sg(t):=\eta(\phi(t))\in \partial \func(u(t))$ and $q(t):=r(\phi(t))\in\partial H(u(t))$, one computes
\begin{align*}
    u'(t) &= \phi'(t)v'(\phi(t)) 
    = \phi'(t)\left[q(t) - \frac{\sg(t)}{H_*(\sg(t))}\right] \\
    &= \frac{\func(u(t))}{H(u(t))} q(t) -  \frac{\func(u(t))}{H(u(t))H_*(\sg(t))}\sg(t) 
    = R(u(t))q(t) - \frac{\func(u(t))}{H(u(t))H_*(\sg(t))}\sg(t).
\end{align*}
\end{proof}
\begin{rem}
By the definition of the dual functional $H_*$ (cf.~\labelcref{eq:dual_semi-norm}) it holds
\begin{align*}
    \frac{\func(u(t))}{H(u(t))H_*(\sg(t))} = \frac{\langle\sg(t),u(t)\rangle}{H(u(t))H_*(\sg(t))} \leq 1,\quad\forall t>0.
\end{align*}
Since, however, \labelcref{eq:asymptotic-nossek-flow} converges to an eigenvector as $t\to\infty$, one can even show
\begin{align*}
    \lim_{t\to\infty}\frac{\func(u(t))}{H(u(t))H_*(\sg(t))} = 1.
\end{align*}
This makes \labelcref{eq:asymptotic-nossek-flow} asymptotically equivalent to \labelcref{eq:FAGP-flow}.
\end{rem}
\begin{rem}
Note that the initial value problem \labelcref{eq:ivp} admits a unique solution, e.g., if ${\func(v(0))}/{H(v(0))}<\infty$.
As \citet{nossek2018flows} one can show that $\frac{\d}{\d t}\frac{\func(v(t))}{H(v(t))}\leq 0$, such that the right hand side of \labelcref{eq:ivp} is a continuous and non-increasing function, for which existence of the initial value problem is guaranteed by the classical Peano theorem and uniqueness follows from the monotonicity of the right hand side.
\end{rem}

\section{Nonlinear Power Methods for Homogeneous Functionals}
\label{sec:power_method}

In this section we first show that the time discretization of the normalized gradient flow \labelcref{eq:normalized_gradflow} yields a nonlinear power method of the proximal operator.
Afterwards, we analyze general nonlinear power methods of $p$-proximal operators in Banach spaces.

A nonlinear power method for the computation of matrix norms has already been investigated in the early work by \citet{boyd1974power}.
\citet{buhler2009spectral,hein2010inverse} applied nonlinear power methods for 1-Laplacian and $p$-Laplacian eigenvectors to graph clustering.
The convergence of $p$-Laplacian inverse power methods for ground states and second eigenfunctions was investigated by \citet{bozorgnia2016convergence,bozorgnia2020approximation}.
Nonlinear power methods and Perron-Frobenius theory for order-preserving multihomogeneous maps were analyzed by \citet{gautier2019perron,gautier2019unifying,gautier2020computing}.
Finally, \citet{bungert2020nonlinear} first investigated proximal power methods, focusing on absolutely one-homogeneous functionals on Hilbert spaces.
In the following, we extend these results to our general setting.

\subsection{Normalized Gradient Flows and Power Methods}
The normalized gradient flow~\labelcref{eq:normalized_gradflow}, analyzed above, already bears a lot of similarity with a power method.
Indeed, discretizing it in time using the minimizing movement scheme \labelcref{eq:minimizing movements} yields a power method for the proximal operator as the following proposition shows.
\begin{prop}\label{prop:normalized_gradflow_powermethod}
Let \revise{$\func:\H \to (-\infty,\infty]$ be an absolutely $\alpha$-homogeneous functional on a Hilbert space $\H$}, and let the sequences $u^k$ and $w^k$ be generated by the iterative scheme
\begin{align}
\begin{cases}
    u^0 &= f, \\
    u^{k+1} &= \prox_{\tau^k\func}(u^{k}),\quad k\in\N_0, \\
    w^{k+1} &= \frac{u^{k+1}-\overline{f}}{\norm{u^{k+1} - \overline{f}}},
\end{cases}
\end{align}
where $(\tau^k)_{k\in\N} \subset \R_+$ is a sequence of step sizes.
Then it holds that 
\begin{align}
    w^{k+1} = \frac{\prox_{\sigma^k}(w^{k})}{\norm{\prox_{\sigma^k}(w^{k})}},
\end{align}
where the step sizes $\sigma^k$ are given by
\begin{align}
    \sigma^k:=\tau^k\norm{u^{k}-\overline{f}}^{\alpha-2}.
\end{align}
\end{prop}
\begin{proof}
Without loss of generality we can assume $\overline{f}=0$.
Using the homogeneity of $\func$, we can compute
\begin{align*}
    u^{k+1} &= \prox_{\tau^k\func}(u^k) 
    = \prox_{\tau^k\func}\left(\norm{u^k}w^k\right) 
    = \argmin_{u\in\H} \frac{1}{2} \norm{u - \norm{u^k}w^k}^2 + \tau^k \func(u) \\
    &= \argmin_{u\in\H} \frac{1}{2} \norm{\frac{u}{\norm{u^k}} - w^k}^2 + \tau^k\norm{u^k}^{-2} \func(u) \\
    &= \argmin_{u\in\H} \frac{1}{2} \norm{\frac{u}{\norm{u^k}} - w^k}^2 + \tau^k\norm{u^k}^{\alpha-2}\func\left(\frac{u}{\norm{u^k}}\right) 
    = \norm{u^k} \prox_{\sigma^k}(w^k).
\end{align*}
This readily implies
\begin{align*}
    w^{k+1} = \frac{u^{k+1}}{\norm{u^{k+1}}} = \frac{\prox_{\sigma^k\func}(w^k)}{\norm{\prox_{\sigma^k\func}(w^k)}}.
\end{align*}
\end{proof}

\subsection{Analysis of Nonlinear Power Methods}
As we have seen above, basically all flows for the computation of nonlinear eigenfunctions are equivalent to the normalized gradient flow~\labelcref{eq:normalized_gradflow}. 
Since the latter gives rise to a nonlinear power method, as shown in \cref{prop:normalized_gradflow_powermethod}, we thoroughly analyze such power methods in larger generality in the following.
Here, we generalize the proof strategy, which was developed by \citet{bungert2020nonlinear} for proximal power methods of absolutely one-homogeneous functionals in Hilbert spaces.
In particular, we utilize a Banach space framework and investigate the power method of the $p$-proximal operator of an absolutely $\alpha$-homogeneous functional.

To avoid heavy notation we pose \cref{ass:hom_nullspace} for the rest of this section, stating that $\func$ is absolutely $\alpha$-homogeneous with trivial nullspace. 
This is a common assumption for eigenvalue problems of homogeneous functionals, see, e.g., \citet{hynd2017approximation,bungert2019asymptotic}, and can always be achieved by replacing the space $\X$ with the quotient space $\X/\calN(\func)$.

\begin{example}[The $p$-Dirichlet energies]
Let us study the case that $\func(u)=\int_\Omega|\nabla u|^p\d x$ equals the $p$-Dirichlet energy for $p>1$ or $\func(u)=\tv(u)$ equals the total variation for $p=1$. 
By extending to $\infty$ these functionals can be defined for $u\in L^p(\Omega)$, however, the appropriate space to make sure $\calN(\func)=\{0\}$ is given by $\X:=L^p(\Omega)/\{u\equiv const.\}$, equipped with the norm $\norm{u}_\X:=\inf_{c\in\R}\norm{u-c}_{L^p(\Omega)}$.
Here, Poincar\'{e}'s inequality makes sure that 
\begin{align*}
    \lambda_p := \inf_{u\in L^p(\Omega)} \frac{\func(u)}{\norm{u}_\X^p}= \inf_{u\in L^p(\Omega)} \frac{\func(u)}{\inf_{c\in\R}\norm{u-c}_{L^p(\Omega)}^p}>0.
\end{align*}
\end{example}
We now analyze the nonlinear power method
\begin{align}\label{eq:power_it_prox}
\begin{cases}
u^{k+1/2}&\in\pprox_{\sigma^k\func}(u^k),\\
u^{k+1}&=\frac{u^{k+1/2}}{\norm{u^{k+1/2}}},
\end{cases}
\end{align}
where $\sigma^k>0$ are regularization parameters that can depend on $u^k$, and $u^0$ is chosen such that $\norm{u^0}=1$. 
\revise{Note that the iteration \labelcref{eq:power_it_prox} chooses an arbitrary point in the (possibly multivalued) $p$-proximal operator which is reminiscent of subgradient descent. 
If the Banach space $\X$ is strictly convex and $p>1$, this choice is unique \citep{schuster2012regularization}.}

For proving convergence, we use similar arguments as \citet{bungert2020nonlinear} where proximal power iterations on Hilbert spaces were investigated.
To this end we first introduce suitable parameter choice rules for the parameter $\sigma^k$ in \labelcref{eq:power_it_prox}, which make sure that the power method is well-defined and we can controll their convergence.
\begin{definition}[Parameter rules]\label{def:parameter_rule}
Let $0<c<1$ be a constant.
\begin{itemize}
    \item The \emph{constant} parameter rule is given by $\sigma^k=\frac{c}{\func(u^0)}$ for all $k\in\N$.
    \item The \emph{variable} parameter rule is given by $\sigma^k=\frac{c}{\func(u^k)}$ for all $k\in\N$.
\end{itemize}
\end{definition}
Our first statement collects two important properties of the proximal power method \labelcref{eq:power_it_prox}, namely that its iterates are normalize and their energy decreases. 
The proof is based on \cref{prop:decrease_rayleigh} and works as the one by \citet{bungert2020nonlinear}.

\begin{prop}\label{prop:decrease_func}
Let \cref{ass:hom_nullspace} hold and $\sigma^k$ be given by a rule in \cref{def:parameter_rule}.
Then for all $k\in\N_0$ the iterative scheme \labelcref{eq:power_it_prox} is well-defined and satisfies
\begin{enumerate}
\item $\norm{u^k}=1$,
\item $\func(u^{k+1})\leq \func(u^k)$.
\end{enumerate}
\end{prop}

Our first statement is concerns the angle between $u^{k+1/2}$ and $u^k$, which necessarily should converge to zero as the proximal power method \labelcref{eq:power_it_prox} converges.
The proof can be found in the appendix.
\begin{prop}[Angular convergence]\label{prop:angle}
The iterates of the proximal power method \labelcref{eq:power_it_prox} satisfy
\begin{align}\label{eq:rate_angle}
    \lim_{k\to\infty}\norm{u^{k+1/2}-u^k}^p - \left|\norm{u^{k+1/2}}-1\right|^p = 0. 
\end{align}
If $\norm{\cdot}$ is a Hilbert norm and $p=2$ this can be simplified to
\begin{align}\label{eq:1-angle}
    \frac{\langle u^{k+1/2},u^k\rangle}{\norm{u^{k+1/2}}\norm{u^k}} =  1,%
\end{align}
which shows that the cosine of the angle converges to 1 and hence the angle to zero.
\end{prop}
Now we can show \revise{subsequential convergence of the proximal power method \labelcref{eq:power_it_prox}} to an eigenvector of the $p$-proximal operator.
The proof which is very similar to \citet{bungert2020nonlinear} can be found in the appendix.
\begin{thm}[Convergence of the proximal power method]\label{thm:cvgc_power_method}
Let \cref{ass:compactness} and \cref{ass:hom_nullspace} be fulfilled and assume that $\func$ satisfies the coercivity condition
\begin{align}\label{eq:coercivity}
    \inf_{u\in\X}\frac{\func(u)}{\norm{u}^\alpha}>0.
\end{align}
Let the sequence $(u^k)_{k\in\N}$ be generated by \labelcref{eq:power_it_prox} with a step size rule from \cref{def:parameter_rule}.
Then it holds
\begin{itemize}
    \item the step sizes $(\sigma^k)_{k\in\N}$ are non-decreasing and converge to $\sigma^*\in(0,\infty)$,
    \item the sequence $(u^k)_{k\in\N}$ admits a subsequence converging to some $u^*\in\X\setminus\{0\}$,
    \item there exists $\mu\in(0,1]$ such that 
    \begin{align}\label{eq:eigenvec_prox}
        \mu u^* \in \pprox_{\sigma^*\func}(u^*).
    \end{align}
\end{itemize}
Furthermore, if $p>1$ it holds $\mu<1$.
\end{thm}
\begin{rem}
For $p=1$ it can happen that $u^*\in\pprox_{\sigma^*\func}(u^*)$ although $u^*\neq 0$ is not a minimizer of $\func$, which is known as exact penalization \citep{bungert2019solution,bungert2020variational}.
It occurs for positive $\sigma^*>0$ smaller than the value $\sigma_*(u^*)$ (cf. \cref{prop:exact_recon}).
The underlying reason that in this case the proximal power iteration may converge to eigenvectors with eigenvalue one is that for $p=1$ the upper bound of the exact reconstruction time \labelcref{ineq:upper_bound_alpha*} coincides with the lower bound for the extinction time \labelcref{ineq:lower_bd_sigma_**} and both become sharp for eigenvectors.
\end{rem}

For $p>1$ eigenvectors of the $p$-proximal are in one-to-one correspondence to eigenvectors in the sense of \labelcref{eq:gen_ev_prob}.
Hence, in this case the limit of the proximal power method is a solution to this nonlinear eigenvalue problem.

\begin{thm}[Convergence to subdifferential eigenvector]\label{thm:cvgc_to_subdiff_ev}
Assume that $p>1$.
Then under the conditions of \cref{thm:cvgc_power_method} there exists $\lambda>0$ and $u^*\in\X$ such that, up to a subsequence, the sequence $(u^k)_{k\in\N}$ generated by \labelcref{eq:power_it_prox} converges to $u^*$ which satisfies
\begin{align*}
    -\lambda\Phi_\X^p(u^*) + \partial\func(u^*) \ni 0.
\end{align*}
\end{thm}

We conclude this section by proving that under generic conditions the proximal power method \labelcref{eq:power_it_prox} preserves positivity. 
In many scenarios ground states in the sense of \cref{def:ground_states} can be characterized as unique non-negative eigenvectors \citep{bungert2020structural,roith2020continuum,boyd1974power,bungert2019asymptotic}, in which case \cref{thm:cvgc_to_subdiff_ev} implies convergence to \revise{the unique} ground state.

\begin{prop}
Assume that $\X$ is a strictly convex Banach lattice and $\func(|u|)\leq \func(u)$ for all $u\in\X$.
If the proximal power method \labelcref{eq:power_it_prox} is initialized with $u^0\geq 0$, then all iterates of \labelcref{eq:power_it_prox} are non-negative and hence also the eigenvector $u^*$ from \cref{thm:cvgc_power_method} is non-negative.
\end{prop}
\begin{proof}
The proof works inductively.
Assuming that $u^k\geq 0$ we would like to show that $u^{k+1/2}\geq 0$ which implies the same for $u^{k+1}$.
To this end, we define a competitor for $u^{k+1/2}$ by setting $\tilde{u}:=({u^{k+1/2}+|u^{k+1/2}|})/{2} \geq 0$.
Using the assumptions, we can calculate
\begin{align*}
    &\phantom{=}\;\frac{1}{p}\norm{\tilde{u}-u^{k}}^p + \sigma^k\func(\tilde{u}) \\
    &=\frac{1}{p}\norm{\frac{u^{k+1/2}+|u^{k+1/2}|}{2}-\frac{u^{k}+|u^{k}|}{2}}^p + \sigma^k\func\left(\frac{u^{k+1/2}+|u^{k+1/2}|}{2}\right) \\
    &\leq \frac{1}{2p}\norm{u^{k+1/2}-u^k}^p + \frac{\sigma^k}{2}\func(u^{k+1/2}) +
    \frac{1}{2p}\norm{|u^{k+1/2}|-|u^k|}^p + \frac{\sigma^k}{2}\func(|u^{k+1/2}|) \\
    &\leq \frac{1}{p}\norm{u^{k+1/2}-u^k}^p + \sigma^k\func(u^{k+1/2}) 
    \leq \frac{1}{p}\norm{u-u^k}^p + \sigma^k\func(u),\qquad\forall u\in\X.
\end{align*}
Hence, $\tilde{u}\in\pprox_{\sigma^k\func}(u^k)$ and using the strict convexity of $\X$ we infer $u^{k+1/2}=\tilde{u}\geq 0$.
\end{proof}

\section{\texorpdfstring{$\Gamma$}{Gamma}-Convergence implies Convergence of Ground States}
\label{sec:gamma-cvgc}

Having studied the both flows and nonlinear power methods which converge to solutions of nonlinear eigenvalue problems, we now study how eigenfunctions behave under approximation.
More precisely, we consider the convergence of ground states of $\Gamma$-converging functionals, as defined in \cref{def:ground_states}.

An important field of application for our results are continuum limits, which study $\Gamma$-convergence of functionals defined on grids or weighted graphs towards their continuum versions as the grids / graphs become denser.
Because of their applicability in graph clustering and other data science tasks, energies depending on the gradient of a function have been investigated a lot in the last years.
E.g., these studies proved continuum limits for the graph and grid total variation \citep{trillos2015continuum,chambolle2020approximating}, the $p$-Dirichlet energies \citep{slepcev2019analysis}, and a functional related to the Lipschitz constant \citep{roith2020continuum}.

While most of these works applied their results to minimization problems featuring the respective functionals, \citet{roith2020continuum} observed that $\Gamma$-convergence of the investigated absolutely one-homogeneous functional implies also convergence of its ground states.
Apart from this paper, there is not much literature on the stability of ground states or more general eigenvector problems under Gamma convergence; the only other references we found is \citet{alvarez2010asymptotic}, which studies linear eigenvalue problems of a specific form and \citet{brasco2015stability} which deal with the convergence of fractional to local $p$-Laplacian eigenfunctions.

In this section we generalize the result from \citet{roith2020continuum} to general convex functionals.
We start with recapping the definition of $\Gamma$-convergence \citep{braides2002gamma} and then prove the convergence result.

\begin{definition}[$\Gamma$-convergence]
A sequence of functionals $(\func_k)_{k\in\N}$ on a metric space $\X$ is said to $\Gamma$-converge to a functional $\func$ on $\X$ (written as $\func_k\gammato\func$) if the following two conditions hold.
\begin{itemize}
    \item \textbf{liminf inequality:} For all sequences $(u_k)_{k\in\N}\subset\X$ which converge to $u\in\X$ it holds
    \begin{align*}
        \func(u) \leq \liminf_{k\to\infty} \func_k(u_k).
    \end{align*}
    \item \textbf{limsup inequality:} For all $u\in\X$ there exists a sequence $(u_k)_{k\in\N}\subset\X$ (called recovery sequence) which converges to $u\in\X$ and satisfies
    \begin{align*}
        \func(u) \geq \limsup_{k\to\infty} \func_k(u_k).
    \end{align*}
\end{itemize}
\end{definition}

We can now prove the main theorem of this section, stating that a convergent sequence of ground states converges to a ground state of the limiting functional.
For this we need to assume that the minimizers of the limiting functional can be approximated by the ones of the approximating functionals in the following distance
\begin{align*}
    \d(A,B) = \sup_{x\in A}\inf_{y\in B}\norm{x-y},\quad A,B\subset\X,
\end{align*}
which constitutes a part of the Hausdorff distance.
Note that this is quite a weak condition and in typical applications, as the ones described above, the set of minimizers of the approximating and limiting functionals even coincide.

\begin{thm}\label{thm:cvgc_ground_states}
Assume that 
\begin{subequations}
\begin{alignat}{2}
    \label{eq:ass_gamma}
    \func_k&\gammato \func,\quad &k\to\infty,\\
    \label{eq:ass_argmins}
    \d(\argmin\func,\argmin\func_k) &\to 0,\quad &k\to\infty.
\end{alignat}
\end{subequations}
Let $(u_k,\hat{u}_k)\subset\X\times\argmin\func_k$ be a sequence of $p$-ground states of $\func_k$ and assume that $u_k\to u_*$ and $\hat{u}_k\to\hat{u}_*$. 
Then $(u_*,\hat{u}_*)\subset\X\times\argmin\func$ is a $p$-ground state of $\func$ and it holds
\begin{align}\label{eq:approx_ground_states}
    \frac{\func(u_*)-\func(\hat{u}_*)}{\norm{u_*-\hat{u}_*}^p} 
    = \lim_{k\to\infty} \frac{\func_k(u_k)-\func_k(\hat{u}_k)}{\norm{u_k-\hat{u}_k}^p}.
\end{align}
\end{thm}
\begin{rem}
A few remarks regarding assumption \labelcref{eq:ass_argmins} are in order.
\begin{itemize}
    \item It is satisfied for all the examples above, where $\argmin\func$ coincides with the set of constant functions and $\func_k$ is a discrete $p$-Dirichlet energy on a graph \citep{slepcev2019analysis,roith2020continuum}.
    Its set of minimizers consists of functions which are constant on each connected component of the graph.
    Hence, one even has the inclusion $\argmin\func\subset\argmin\func_k$ which implies $\d(\argmin\func,\argmin\func_k) = 0$.
    \item For general $\Gamma$-converging functionals \labelcref{eq:ass_argmins} is wrong, e.g., for $f_k(x)=|x|^k$ with $\argmin\func_k=\{0\}$ and $f(x)=\chi_{[-1,1]}(x)$ with $\argmin\func=[-1,1]$.
    \item If $\func$ has a unique minimizer, meaning $\argmin\func$ is a singleton, and $\func_k$ is a sequence of equi-coercive functionals (see~\citet{braides2002gamma} for definitions) condition~\labelcref{eq:ass_argmins} is also satisfied.
    In this case any sequence of minimizers $(u_k)$ of $\func_k$ converges to a minimizer of $\func$.
    If this minimizer is uniquely determined it can therefore be approximated with elements $u_k\in\argmin\func_k$.
\end{itemize}
\end{rem}
\begin{proof}
By \labelcref{eq:ass_gamma} it holds $\hat{u}_*\in\argmin\func$. 
Also, $(\hat{u}_k)_{k\in\N}$ is a recovery sequence of~$\hat{u}_*$.
To see this, take another recovery sequence $(\tilde{v}_k)_{k\in\N}$ of $\hat{u}_*$.
Since $\hat{u}_k\in\argmin\func_k$, it holds
\begin{align*}
    \limsup_{k\to\infty}\func_k(\hat{u}_k) \leq \limsup_{k\to\infty}\func_k(\tilde{v}_k) \leq \func(\hat{u}_*),
\end{align*}
which means that $(\hat{u}_k)_{k\in\N}$ is a recovery sequence of $\hat{u}_*$.
Using the liminf inequality we get 
\begin{align*}
    \lim_{k\to\infty}\func_k(\hat{u}_k)=\func(\hat{u}_*).
\end{align*}
Let now $w\in\X$ be arbitrary and $(w_k)_{k\in\N}$ be a recovery sequence for $w$.
Using the assumptions we can compute
\begin{align*}
    \frac{\func(u_*)-\func(\hat{u}_*)}{\norm{u_*-\hat{u}_*}^p} 
    \overset{\labelcref{eq:ass_gamma}}{\leq} &\liminf_{k\to\infty} \frac{\func_k(u_k)-\func_k(\hat{u}_k)}{\norm{u_k-\hat{u}_k}^p} \\
    \overset{\labelcref{eq:ground_states}}{\leq} &\liminf_{k\to\infty} \frac{\func_k(w_k)-\func_k(\hat{u}_k)}{\norm{w_k-\hat{u}_k}^p} 
    \overset{\labelcref{eq:ass_gamma}}{\leq} \frac{\func(w)-\func(\hat{u}_*)}{\norm{w-\hat{u}_*}^p},
\end{align*}
hence, $u_*$ solves the minimization problem in \labelcref{eq:ground_states}.
Choosing $w=u_*$ shows \labelcref{eq:approx_ground_states}.

Similarly, let $\hat{w}\in\argmin\func$ be arbitrary and let $(w_k)_{k\in\N}$ be a recovery sequence for~$\hat{w}$.
By assumption \labelcref{eq:ass_argmins} there exists a sequence $(\hat{w}_k)_{k\in\N}\subset\X$ such that $\hat{w}_k\in\argmin\func_k$ for all $k\in\N$ and $\norm{\hat{w}-\hat{w}_k}\to 0$ as $k\to\infty$.
We claim that $(\hat{w}_k)_{k\in\N}$ is also a recovery sequence for $\hat{w}$.
To see this we first we observe that---since $\hat{w}_k\in\argmin\func_k$---it holds
\begin{align*}
    \limsup_{k\to\infty}\func_k(\hat{w}_k) \leq \limsup_{k\to\infty}\func_k(w_k) \leq \func(\hat{w}). 
\end{align*}
Hence, $(\hat{w}_k)_{k\in\N}$ is also a recovery sequence for $\hat{w}$ and we can compute
\begin{align*}
    \frac{\func(u_*)-\func(\hat{u}_*)}{\norm{u_*-\hat{u}_*}^p} 
    \overset{\labelcref{eq:approx_ground_states}}{=} &\lim_{k\to\infty} \frac{\func_k(u_k)-\func_k(\hat{u}_k)}{\norm{u_k-\hat{u}_k}^p} \\
    \overset{\labelcref{eq:ground_states}}{\geq} &\limsup_{k\to\infty} \frac{\func_k(u_k)-\func_k(\hat{w}_k)}{\norm{u_k-\hat{w}_k}^p} 
    \overset{\labelcref{eq:ass_gamma}}{\geq} \frac{\func(u_*)-\func(\hat{w})}{\norm{u_*-\hat{w}}^p},
\end{align*}
hence $\hat{u}_*$ solves the maximization problem in \labelcref{eq:ground_states} and $(u_*,\hat{u}_*)$ is a $p$-ground state.
\end{proof}

\begin{rem}[Compactness and Normalization]
To be able to apply \cref{thm:cvgc_ground_states} one needs some additional compactness and suitable normalization which ensure that ground states of $\func_k$ converge.
For instance, if $\func_k$ and $\func$ are absolutely $p$-homogeneous, arbitrary multiples of $p$-ground states are again $p$-ground states, which means that a-priori a sequence of ground states does not need to converge.
However, in this case one can restrict oneself to normalized ground states solving
\begin{align*}
    u_*\in\argmin\left\lbrace \func(u) \st u\in\X,\,\norm{u}_\sim=1\right\rbrace,
\end{align*}
where $\norm{\cdot}_\sim$ is a quotient norm, defined as $\norm{u}_\sim := \inf_{\hat{u}\in\calN(\func)}\norm{u-\hat{u}}$.
Using the limsup-inequality one can then show as \citet{roith2020continuum} that $\limsup_{k\to\infty}\func_k(u_k) \leq \func(u_*)$, where $u_k$ and $u_*$ denote $p$-ground states of $\func_k$ and $\func$, respectively.
Hence, in this case a compactness assumption of the kind
\begin{align*}
    \limsup_{k\in\N} \left\lbrace\norm{u_k} + \func_k(u_k)\right\rbrace < \infty \implies (u_k)_{k\in\N} \text{ is relatively compact}
\end{align*}
implies the convergence of normalized $p$-ground states using \cref{thm:cvgc_ground_states}.
\end{rem}
Besides the applications on grids and weighted graphs, where the conditions of \cref{thm:cvgc_ground_states} have been verified in the papers referenced at the beginning of this section, one can also apply the theorem for Galerkin approximations using Finite Elements.
\begin{example}[Galerkin discretizations]
Let $\func:\X\to(-\infty,\infty]$ be a continuous functional.
Assume that $(\X_k)_{k\in\N}\subset\X$ is a sequence of embedded approximation spaces such that for all $u\in\X$ there exists a sequence $(u_k)_{k\in\N}\subset\X$ such that $u_k\in\X_k$ for all $k\in\N$ and $u_k\to u$ as $k\to\infty$, in other words it holds $\d(\X,\X_k)\to 0$ as $k\to\infty$.
One can define the functionals
\begin{align}\label{eq:galerkin_general}
    \func_k:\X\to(-\infty,\infty],\quad u\mapsto
    \begin{cases}
    \func(u) & \text{if } u\in\X_k,\\
    \infty & \text{else},
    \end{cases}
\end{align}
which clearly $\Gamma$-converge to $\func$ and hence satisfy \labelcref{eq:ass_gamma}.

The prototypical example for this is a Finite Element Galerkin approximation of the base space $\X=W^{1,p}_0(\Omega)$ with approximating spaces $\X_k$ containing piecewise polynomials.
Interesting functionals in this case are integral functional of the form
\begin{align}\label{eq:galerkin_func}
    \func(u) = \int_\Omega \Phi\left(x,u(x),\nabla u(x)\right) \d x,\quad u \in W^{1,p}_0(\Omega).
\end{align}
If the spaces $\X_k$ contain piecewise polynomials of high degree or if $\Phi$ is strongly non-linear, the functional $\func(u)$ for $u\in\X_k$ cannot be evaluated accurately.
In this case, one can also study $\Gamma$-convergence of quadrature approximations of \labelcref{eq:galerkin_func} as done by \citet{ortner2004gamma}.
\end{example}
\section{Applications}
\label{sec:applications}

In this section we present numerical results of the proximal power method \labelcref{eq:power_it_prox} applied to ground state problems on grids and weighted graphs. 
A weighted graph is a tuple $G=(\widehat{\Omega},\omega)$, where $\widehat{\Omega}$ is a finite set of vertices and $\omega:\widehat{\Omega}\times \widehat{\Omega}\to[0,\infty)$ is a function which assigns edge weights to pairs of vertices.
Weighted graphs are very handy objects for approximating variational problems since they allow for natural definitions of differential operators \citep{elmoataz2015p} and facilitate discrete to continuum analysis using $\Gamma$-convergence or PDE techniques (see, e.g., \citet{trillos2015continuum,calder2018game,slepcev2019analysis,bozorgnia2020infinity,roith2020continuum}). 
We define the set of vertex and edge functions on a graph $G$ as
\begin{align}
    \H(\widehat{\Omega}) = \{u:\widehat{\Omega}\to\R\},\qquad
    \H(\widehat{\Omega}\times \widehat{\Omega}) = \{h:\widehat{\Omega}\times \widehat{\Omega}\to\R\},
\end{align}
together with $p$-norms
\begin{alignat}{3}
    \norm{u}_p^p &:= 
    \sum_{x\in \widehat{\Omega}} |u(x)|^p, \quad
    &&\norm{u}_\infty :=\max_{x\in \widehat{\Omega}} |u(x)|,\quad &&u\in\H(\widehat{\Omega}),\\
    \norm{h}_p^p &:= 
    \sum_{
    x,y\in \widehat{\Omega}}|h(x,y)|^p, \quad
    &&\norm{h}_\infty := \max_{x,y\in \widehat{\Omega}}|h(x,y)|,\quad &&h\in\H(\widehat{\Omega}\times\widehat{\Omega}).
\end{alignat}
Furthermore, one can define a gradient operator $\nabla_G : \H(\widehat{\Omega}) \to \H(\widehat{\Omega}\times \widehat{\Omega})$ and a divergence operator $\div_G : \H(\widehat{\Omega}\times \widehat{\Omega}) \to \H(\widehat{\Omega})$ as
\begin{align}
    \nabla_G u(x,y) &:= \sqrt{\omega(x,y)}(u(y)-u(x)), \\
    \div_G h(x) &:= \sum_{y\in \widehat{\Omega}}\sqrt{\omega(x,y)}(h(y,x)-h(x,y)).
\end{align}
For illustrating the convergence of the proximal power method for the special case of absolutely one-homogeneous functionals $\func$ on a Hilbert space, we define an eigenvector affinity as
\begin{align}\label{eq:affinity}
    \mathrm{aff}_k:=\frac{\norm{\sg^k}^2}{\func(\sg^{k})},\qquad\sg^k:=\frac{u^{k}-u^{k+1/2}}{\sigma^k}.
\end{align}
Note that by the optimality conditions of the 2-proximal operator on a Hilbert space one can easily show that $\mathrm{aff}_k\in[0,1]$ and $\mathrm{aff}_k=1$ if and only if $\sg^k$ is an eigenvector of the subdifferential $\partial\func$ (and hence also of the proximal operator).
\revise{We accurately evaluate all proximal operators using a primal-dual algorithm \citep{chambolle2011first}.}

\subsection{Calibrable Sets}

In this first example the considered graph $G=(\widehat{\Omega},\omega)$ is just a standard rectangular grid, where the differential operators defined above simply coincide with finite difference approximations.
To be precise, $\widehat{\Omega}=\{(ih,ih)\st i=0,\dots,N\}$ where $N\in\N$ and $h=1/N$. 
We consider the discrete total variation functional with central differences
\begin{align}
    \func_G(u) := \frac{1}{2h}\sum_{x\in\widehat{\Omega}} \sqrt{\left[u(x+he_1)-u(x-he_1)\right]^2+\left[u(x+he_2)-u(x-he_2)\right]^2},
\end{align}
where $e_i$ is the $i$-th unit vector (see \citet{chambolle2020approximating} for a plethora of other, in particular, better discretizations). 
We compute the $p$-proximal power iteration~\labelcref{eq:power_it_prox} of $\func_G$ for $p\in\{1,2\}$, using the norms $\norm{\cdot}_1$ and $\norm{\cdot}_2$, respectively.
The first five iterates of the method are depicted in \cref{fig:calibrable}. 
As predicted by \cref{thm:cvgc_power_method}, the power methods converge to eigenvectors of the proximal operators. 
However, only in the case $p=2$, where we used the Hilbert norm for the proximal operator, the limit is a calibrable set \citep{alter2005characterization}, which is an eigenvector of the subdifferential operator in the sense of \labelcref{eq:gen_ev_prob}.
Hence, \cref{thm:cvgc_to_subdiff_ev}, which states that only for $p>1$ proximal eigenvectors are also subdifferential eigenvectors, is sharp, in general.

\begin{figure}
\setlength{\tabcolsep}{1.5pt}
\gdef\Width{2.6cm}
\centering
\begin{tabular}{cccccc}
\rotatebox{90}{\hspace{.7cm}$p=1$} 
&\includegraphics[width=\Width,trim=1cm 1cm 1cm 1cm,clip]{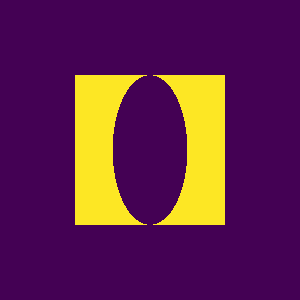} 
&\includegraphics[width=\Width,trim=1cm 1cm 1cm 1cm,clip]{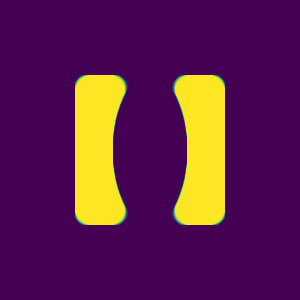} 
&\includegraphics[width=\Width,trim=1cm 1cm 1cm 1cm,clip]{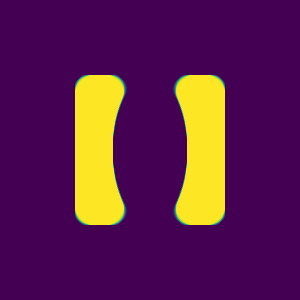} 
&\includegraphics[width=\Width,trim=1cm 1cm 1cm 1cm,clip]{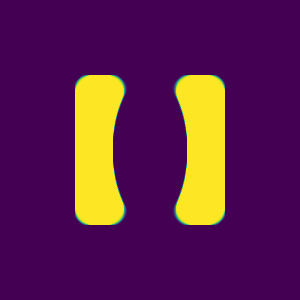} 
&\includegraphics[width=\Width,trim=1cm 1cm 1cm 1cm,clip]{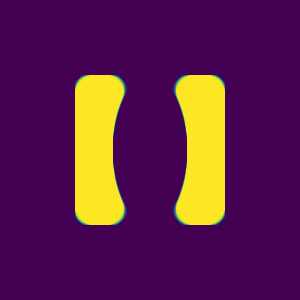} 
\\
& $k=0$ & $k=1$ & $k=2$ & $k=3$ & $k=4$ \\
\rotatebox{90}{\hspace{.7cm}$p=2$} 
&\includegraphics[width=\Width,trim=1cm 1cm 1cm 1cm,clip]{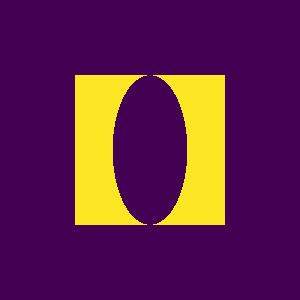} 
&\includegraphics[width=\Width,trim=1cm 1cm 1cm 1cm,clip]{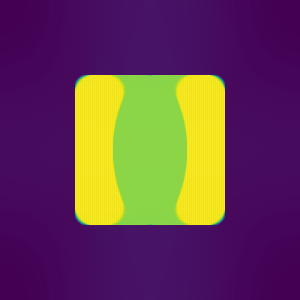} 
&\includegraphics[width=\Width,trim=1cm 1cm 1cm 1cm,clip]{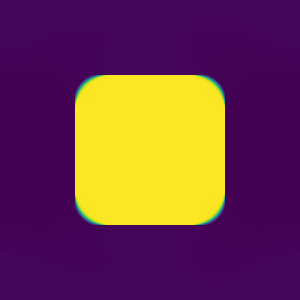} 
&\includegraphics[width=\Width,trim=1cm 1cm 1cm 1cm,clip]{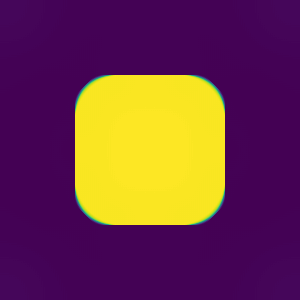} 
&\includegraphics[width=\Width,trim=1cm 1cm 1cm 1cm,clip]{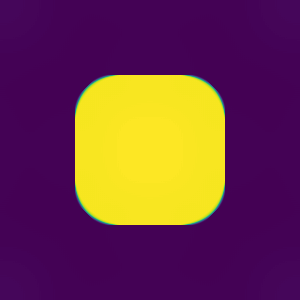} 
\end{tabular}
\caption{\textbf{Top row:} Five iterates for $p=1$. Limit is no calibrable set. \textbf{Bottom row:} Five iterates for $p=2$. Limit is a calibrable set.}
\label{fig:calibrable}
\end{figure}


\subsection{Graph Clustering}

Here, we consider graph clustering using ground states of the graph $p$-Dirichlet energy
\begin{align}\label{eq:graph_diri_en}
    \func_{G,p}(u):=\norm{\nabla_G u}_p^p = \sum_{x,y\in \widehat{\Omega}}\omega(x,y)^\frac{p}{2}|u(y)-u(x)|^p,
\end{align}
where the case $p=1$ is the graph total variation.
If $G$ is connected, the nullspace of these functionals coincides with the constant graph functions.
Ground states are solutions to
\begin{align}
    \min_{u\in\calN(\func_{G,p})^\perp} \frac{\func_{G,p}(u)}{\norm{u}_2},
\end{align}
and correspond to eigenfunctions of the graph $p$-Laplacian operator
\begin{align}
    \Delta_{G,p}u(x):=\sum_{y\in \widehat{\Omega}}\omega(x,y)^\frac{p}{2}|u(y)-u(x)|^{p-2}(u(y)-u(x)).
\end{align}
On the left side of \cref{fig:graph_clustering} we show a graph representing the so-called ``two moons'' data set together with its ground states for $p=2$ and $p=1$.
As observed by \citet{buhler2009spectral,aujol2018theoretical,bungert2019computing} the ground state for $p=1$ yield a sharp clustering of the graph into two sub-graphs whereas the standard graph Laplacian eigenfunction for $p=2$ yields a very diffuse interface.

On the right side of \cref{fig:graph_clustering} we plot the convergence of the proximal power method \labelcref{eq:power_it_prox} for constant (solid lines) and variable parameter rule (dashed lines), cf. \cref{def:parameter_rule}.
The blue lines represent the quantity $1-\cos\theta$ where $\theta$ is the angle between $u^{k+1/2}$ and $u^k$, see \labelcref{eq:1-angle}.
\cref{prop:angle} states that this quantity converges to zero which is reflected in our numerical experiments.
The red lines depict the eigenvector affinity \labelcref{eq:affinity} which converges to one, as desired. 
Notably, here the variable parameter rule leads to a quicker convergence than the constant one.
\cref{fig:graph_clustering} makes clear that after only six iterations the proximal power method has converged to the 1-Laplacian ground state, depicted on the bottom left.

\begin{figure}[htb]
    \begin{minipage}{0.46\textwidth}
    \newlength{\PicWidth}
    \setlength{\PicWidth}{\textwidth}%
    {\includegraphics[angle=0,width=\dimexpr0.5\PicWidth\relax,trim=3cm 4cm 2.5cm 3.5cm,clip]{numerics/1-laplacian efs/two_moon.pdf}}%
    {\includegraphics[angle=0,width=\dimexpr0.5\PicWidth\relax,trim=3cm 4cm 2.5cm 3.5cm,clip]{numerics/1-laplacian efs/laplacian.pdf}}\\%
    \centering
    {\includegraphics[angle=0,width=\dimexpr0.5\PicWidth\relax,trim=3cm 4cm 2.5cm 3.5cm,clip]{numerics/1-laplacian efs/1-laplacian.pdf}}%
    \end{minipage}%
    \hfill%
    \begin{minipage}{0.43\textwidth}
    \includegraphics[width=\textwidth,trim=1.5cm 0.8cm 0.6cm 1.5cm,clip]{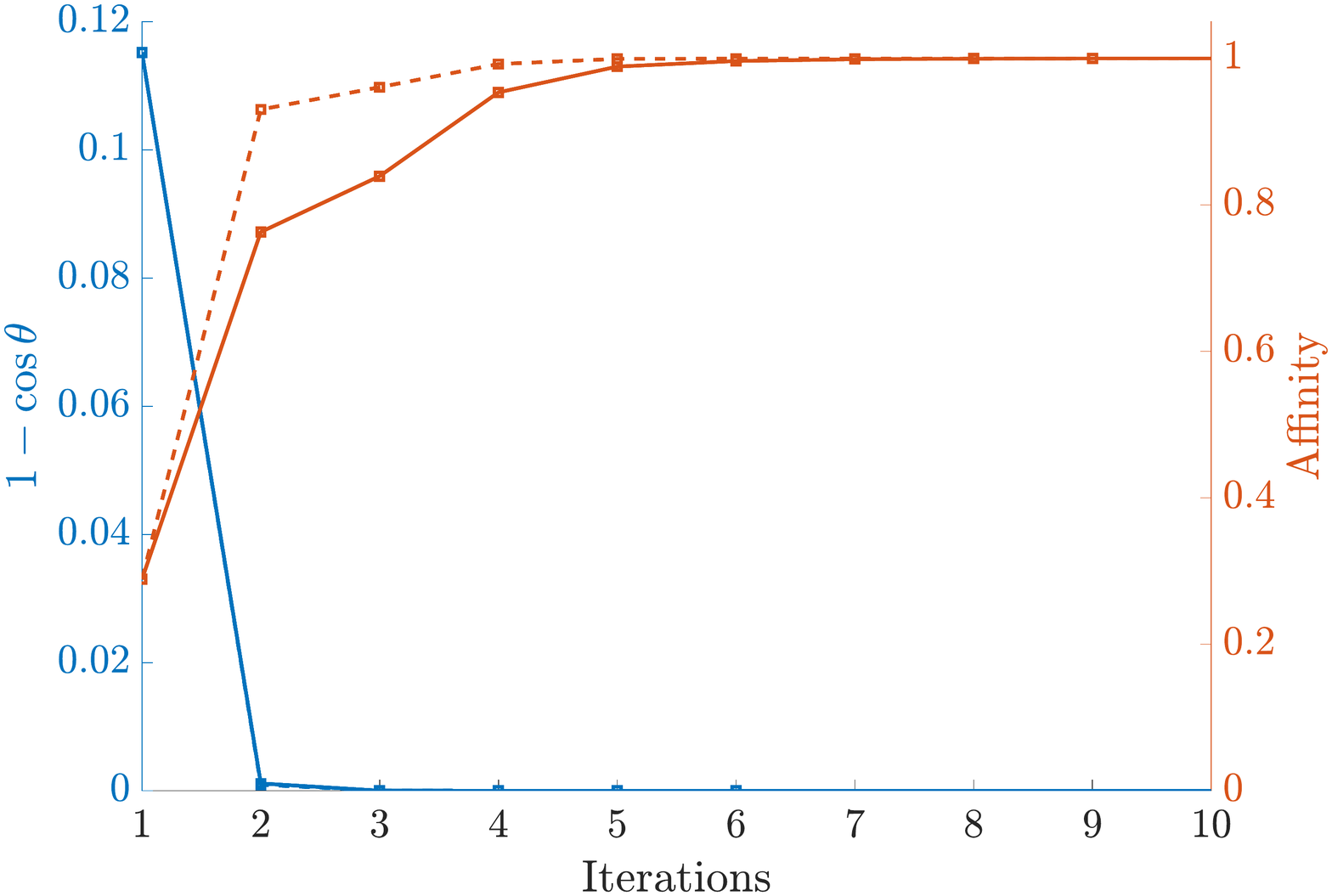}%
    \end{minipage}
    \caption{Spectral graph clustering. \textbf{Left:} Two-moon graph, Laplacian eigenfunction, $1$-Laplacian eigenfunction. \textbf{Right:} Angle (blue) and the eigenvector affinity (red) plotted over the iterations. Solid lines correspond to constant parameter rule, dashed lines to variable.}
    \label{fig:graph_clustering}
\end{figure}

\subsection{Geodesic Distance Functions on Graphs}

In this example we present an application to distance functions on graphs. 
To this end, we choose a subset $\widehat{\mathcal{O}}\subset\widehat{\Omega}$ of the vertex set which play the role of a constraint set.
We define the functional $\func_{G,\infty}:\H(\widehat{\Omega})\to\R\cup\{\infty\}$ as
\begin{align}
    \func_{G,\infty}(u) := \norm{\nabla_G u}_\infty=\max_{x,y\in\widehat{\Omega}}\sqrt{\omega(x,y)}|u(y)-u(x)|
\end{align}
if $u\in\H(\widehat{\Omega})$ satisfies $u=0$ on $\widehat{\mathcal{O}}$ and $\func_{G,\infty}(u)=\infty$ else.
The value $\func_{G,\infty}(u)$ can be interpreted as largest local Lipschitz constant of $u$ which satisfies the constraints on $\widehat{\mathcal{O}}$.
Since the nullspace of $\func_{G,\infty}$ is trivial when the graph is connected, ground states solve
\begin{align}
    \min_{u\in\H(\widehat{\Omega})}\frac{\func_{G,\infty}(u)}{\norm{u}_2}.
\end{align}
\citet{bungert2020structural} characterized solutions of this problem as multiples of the geodesic graph distance function to the set $\widehat{\mathcal{O}}$.
Furthermore, \citet{roith2020continuum} proved that, if the graph vertices $\widehat{\Omega}$ and $\widehat{\mathcal{O}}$ converge to continuum domain $\Omega$ and a closed constraint set $\mathcal{O}\subset\overline{\Omega}$ in the Hausdorff distance and the weights $\omega$ are scaled appropriately, the $\Gamma$-limit of the functionals $\func_{G,\infty}$ is given by
\begin{align}
    \func_\infty(u) = 
    \begin{cases}
    \norm{\nabla u}_{L^\infty(\Omega)},&$if $ u\in W^{1,\infty}(\Omega),\; u=0 $ on $\mathcal{O},\\
    \infty, &$else.$
    \end{cases}
\end{align}
It was proved that the corresponding ground states converge (cf.~\cref{thm:cvgc_ground_states}), as well.
The limiting ground states of $\func_{\infty}$ where also characterized as geodesic distance functions to $\mathcal{O}$ by \citet{bungert2020structural,roith2020continuum}.

\cref{fig:germany,fig:beans} show ground states of $\func_{G,\infty}$ on two types of weighted graphs.
The first one in \cref{fig:germany} is a grid graph representing a map of Germany.
The constraint set $\widehat{\mathcal{O}}$ is chosen as the boundary and hence the computed ground state coincides with the distance function to the border of Germany. 
Similarly, in \cref{fig:beans} we show the computed ground states for different resolutions of a graph which consists of random samples from a two-dimensional dumbbell-shaped manifold, embedded in $\R^3$. 
The constraint set $\widehat{\mathcal{O}}$ is chosen to be a fixed vertex in the top left of the manifold and hence the computed ground states coincide with the geodesic distance to this point.

\begin{wrapfigure}{r}{0.61\textwidth}
    \centering
    \includegraphics[height=0.22\textwidth,trim=6cm 8.7cm 5.5cm 8.3cm,clip]{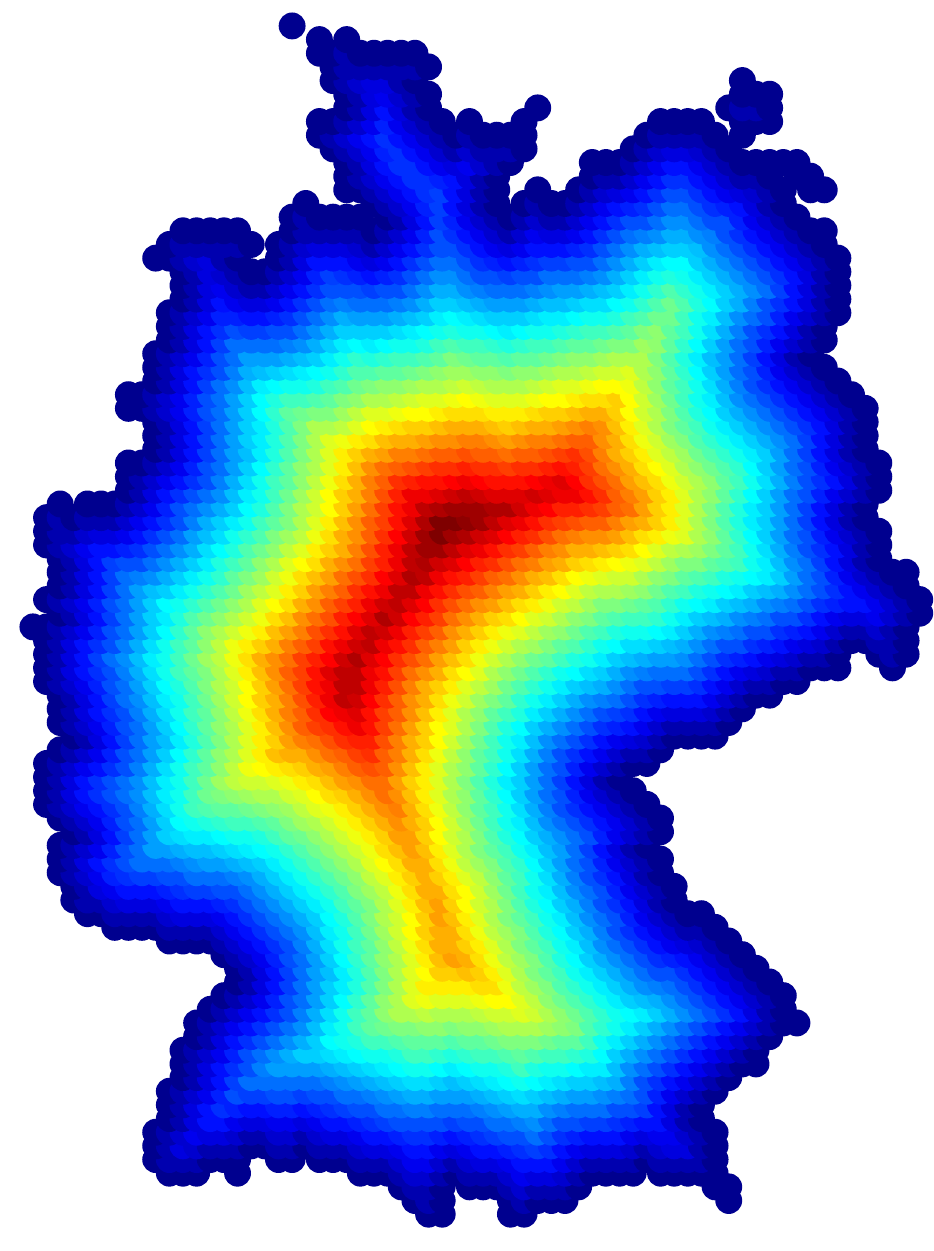}%
    \hfill%
    \includegraphics[width=0.43\textwidth,trim=1.5cm 0.8cm 0.7cm 1.5cm,clip]{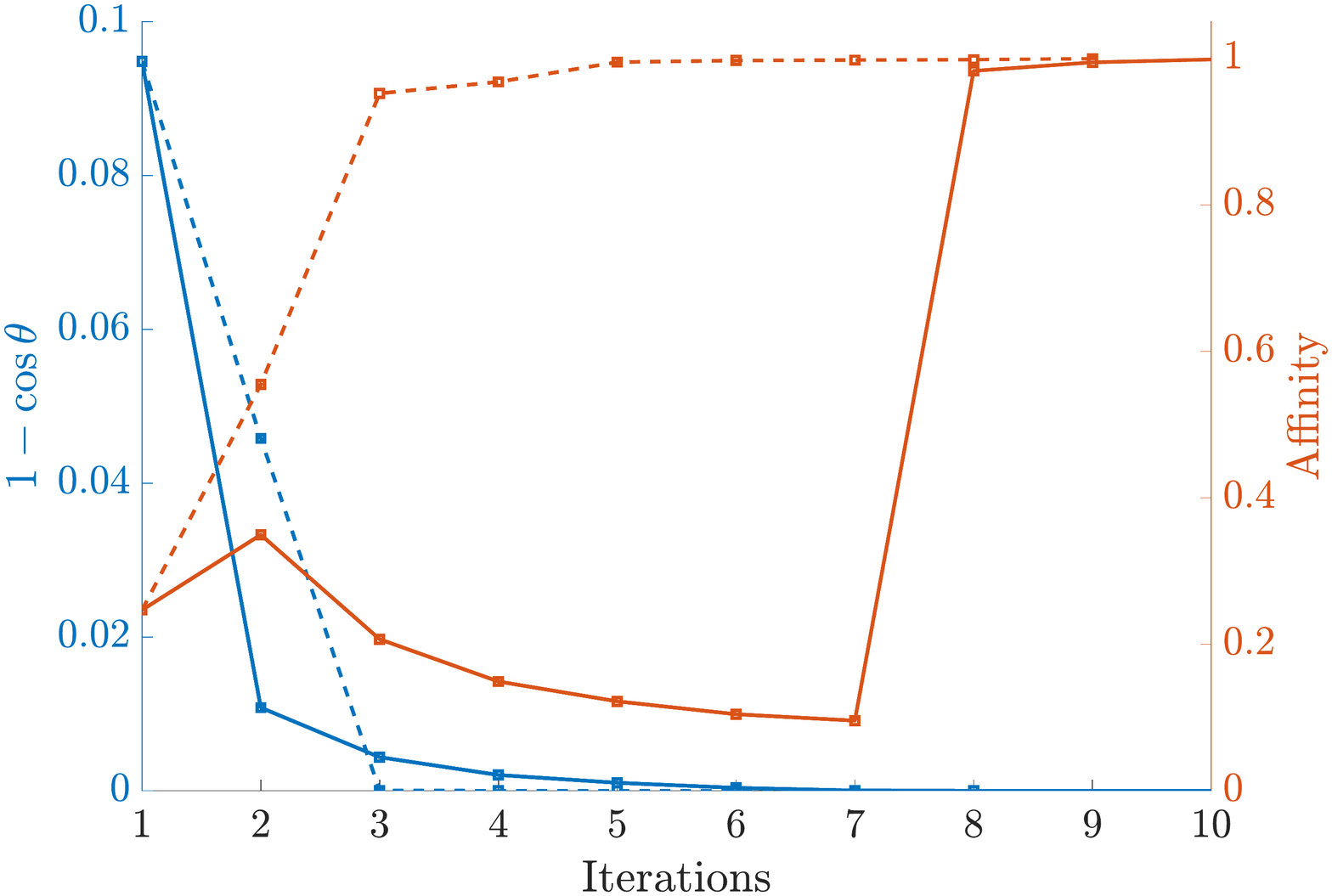}
    \caption{\textbf{Left:} Distance function of the German border. \textbf{Right:} Angle (blue) and the eigenvector affinity (red) plotted over the iterations. Solid lines correspond to constant parameter rule, dashed lines to variable.}
    \label{fig:germany}
\end{wrapfigure}
On the right side of \cref{fig:germany} we again plotted the quantitative convergence behavior of the proximal power method \labelcref{eq:power_it_prox} in terms of convergence of the angle \labelcref{eq:1-angle} and the eigenvector affinity \labelcref{eq:affinity}.
Here we make similar observations as in \cref{fig:graph_clustering}, in particular the variable parameter rule leads to a quicker convergence than the constant one and both the angle and the eigenvector affinity converge after approximately six iterations of the method.
Interestingly, for the constant parameter rule the eigenvector affinity first decreases before it finally reaches a value close to one, whereas the angle decreases monotonously. 
This is no numerical artefact but is due to the fact that the affinity \labelcref{eq:affinity} is a much stronger convergence criterion than the angle since the former depends on the functional at hand (here $\func_{G,\infty}$) whereas the latter does not.

\begin{figure}[b!]
    \def\PicWidth{0.35\textwidth}
    \def\angle{45}
    \centering
    {\includegraphics[angle=\angle,width=\PicWidth,trim=4cm 5cm 3cm 4.5cm,clip]{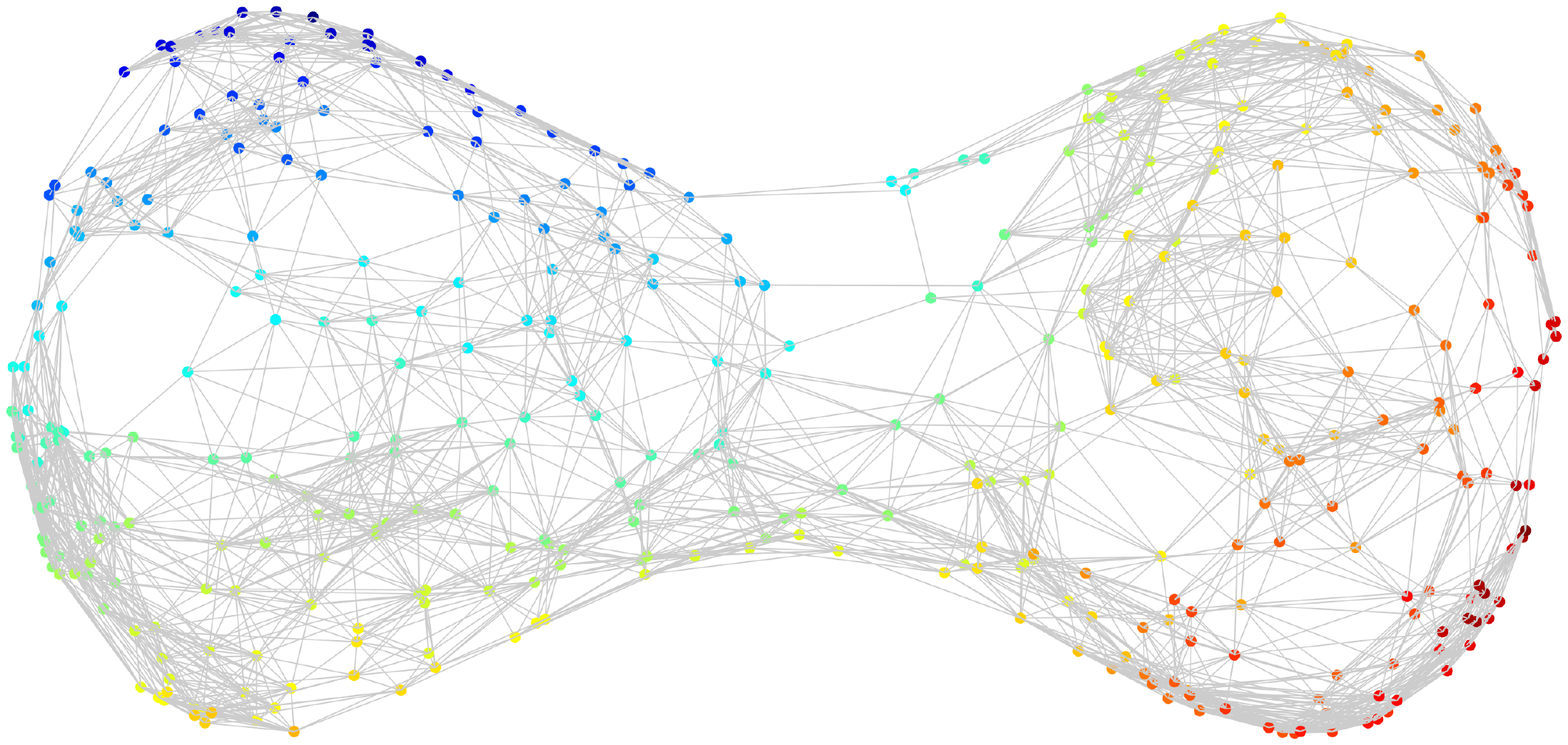}}%
    \hspace{-2cm}%
    {\includegraphics[angle=\angle,width=\PicWidth,trim=4cm 5cm 3cm 4.5cm,clip]{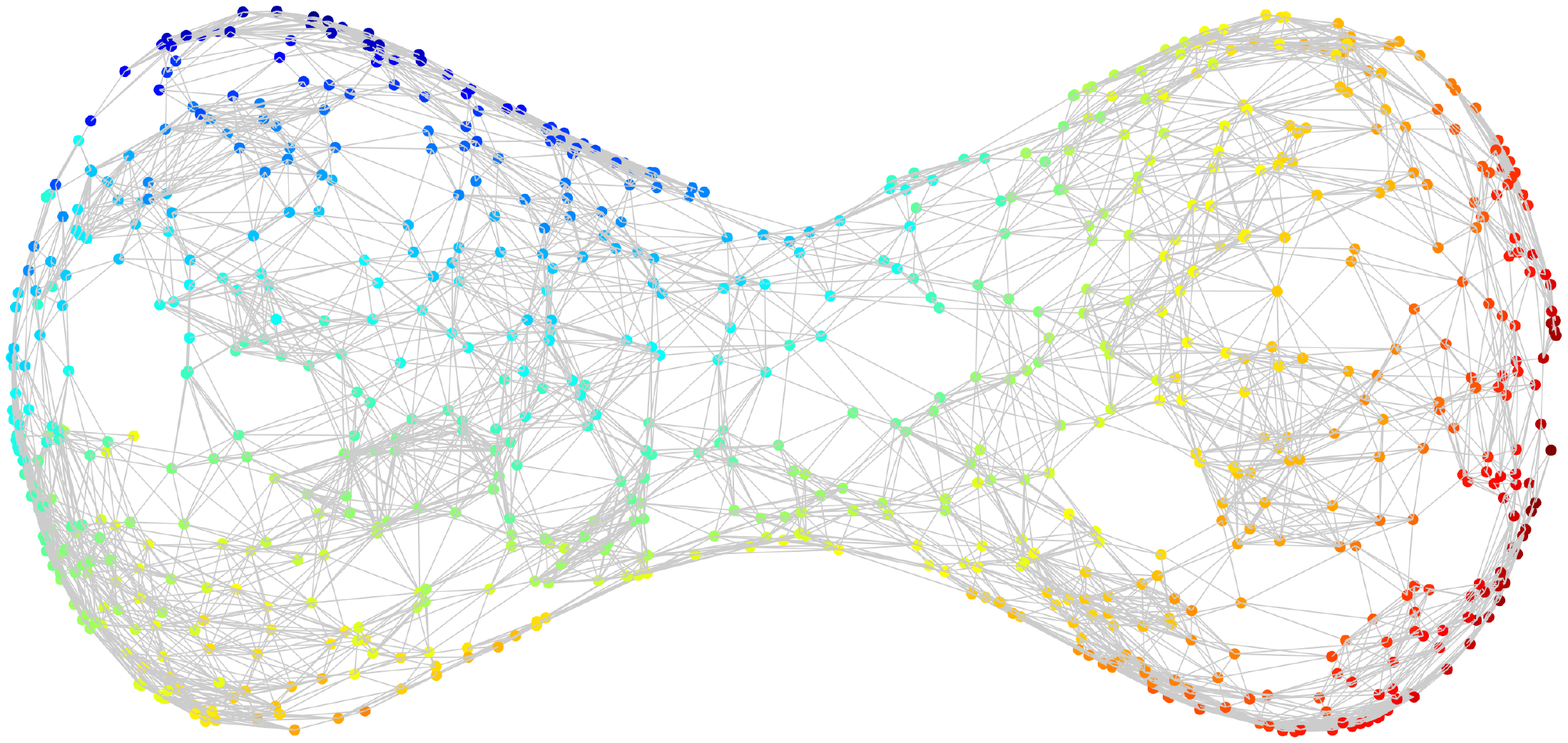}}%
    \hspace{-2cm}%
    {\includegraphics[angle=\angle,width=\PicWidth,trim=4cm 5cm 3cm 4.5cm,clip]{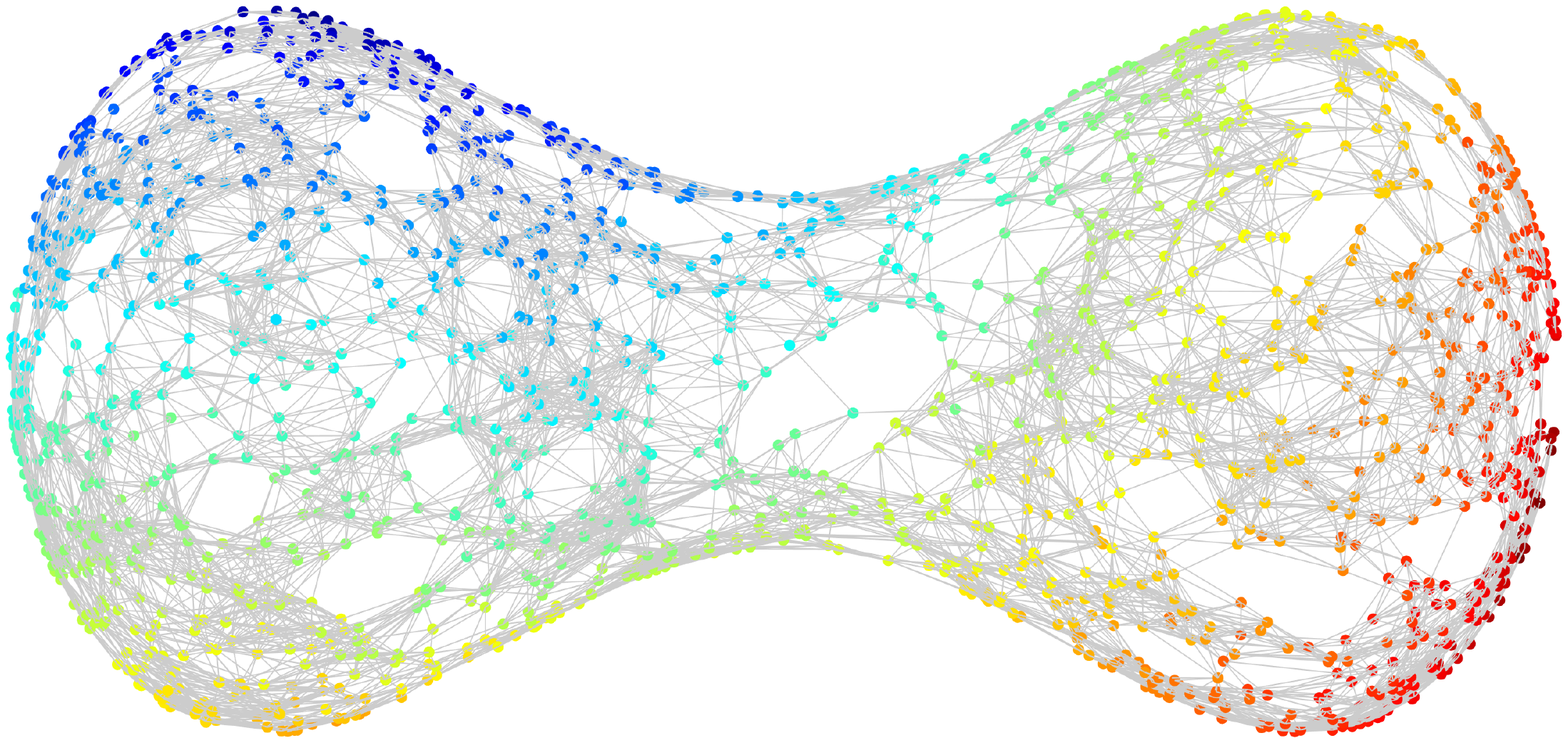}}%
    \hspace{-2cm}%
    {\includegraphics[angle=\angle,width=\PicWidth,trim=4cm 5cm 3cm 4.5cm,clip]{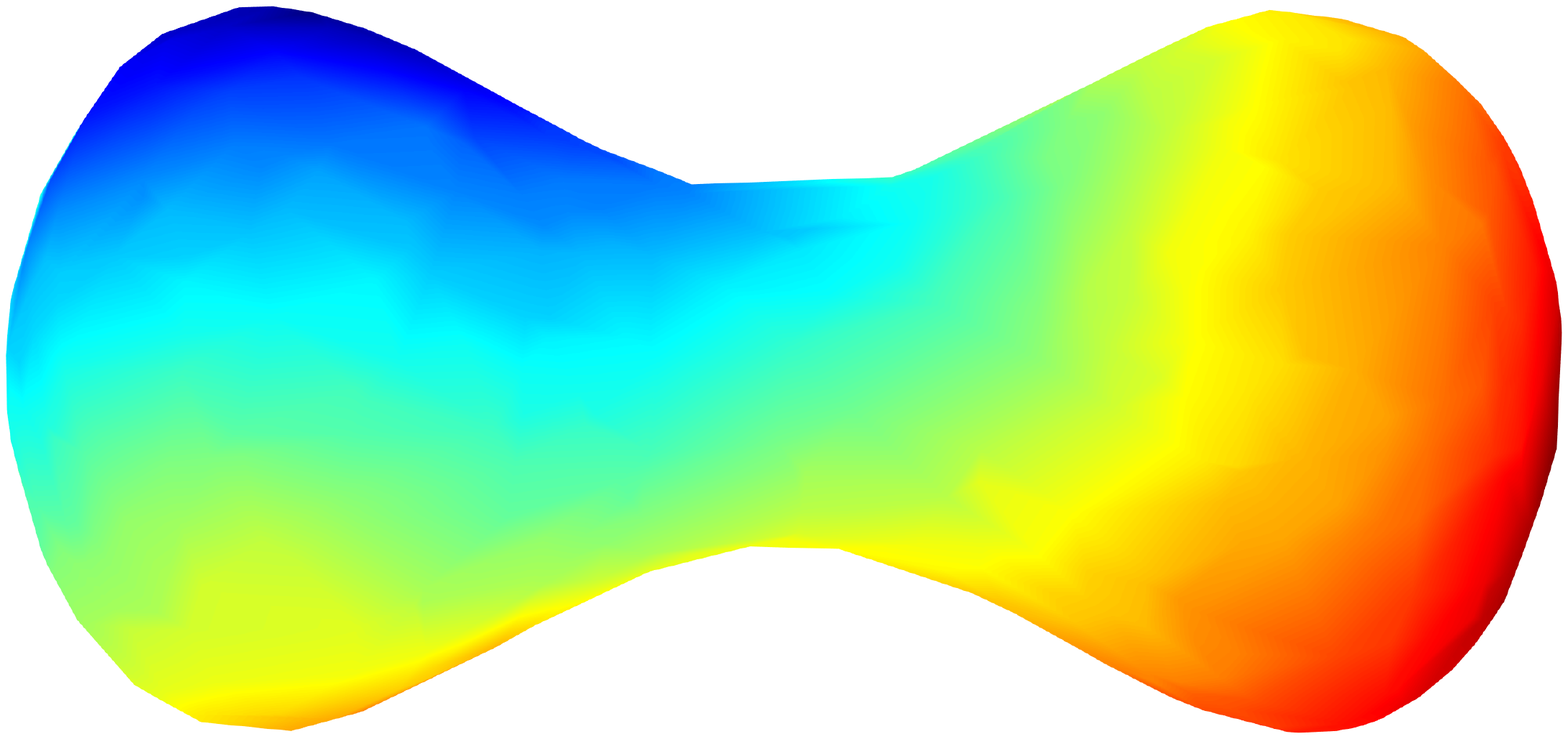}}%
    \caption{Geodesic distance functions to a point at the left of a discretized manifold with increasing resolution. Last image is a rendered version of the penultimate one.}
    \label{fig:beans}
\end{figure}

\clearpage
\printbibliography

\begin{appendices}

The statements regarding exact reconstruction and finite extinction proved below are generalizations of \citet{bungert2019solution}.

\section{Exact Reconstruction Time}

\begin{prop}\label{prop:exact_recon}
It holds that 
\begin{align*}
    f \in \argmin_{u\in\X} \norm{u-f} + \sigma \func(u)
\end{align*}
if and only if $\sigma\leq \sigma_*(f)$, where 
\begin{align}\label{eq:alpha_*}
    \sigma_{*}(f):= \frac{1}{\inf\left\lbrace \norm{\sg}_* \st \sg \in \partial\func(f) \right\rbrace}.
\end{align}
Furthermore, it holds
\begin{align} \label{ineq:upper_bound_alpha*}
   \sigma_*(f) \leq \inf_{\hat{u} \in \argmin\func} \frac{\norm{f-\hat{u}}}{\func(f)-\func(\hat{u})}.
\end{align}
\end{prop}
\begin{proof}
For the first impliction we assume that $f$ solves the minimization problem. 
Then the optimality condition~\labelcref{eq:OC} implies that there is $\eta\in\Phi^1_\X(0)$ and $\sg\in\partial\func(f)$ such that $0=\eta+\sigma\sg$. 
Using that $\Phi^1_\X(0)=\{\eta\in\X^* \st \norm{\eta}_* \leq 1\}$, we obtain
\begin{align*}
    \sigma = \frac{\norm{\eta}_*}{\norm{\sg}_*} \leq  \frac{1}{\inf\left\lbrace \norm{\sg}_* \st \sg \in \partial\func(f) \right\rbrace} = \sigma_*(f).
\end{align*}
Conversely, assume that $\sigma\leq\sigma_*(f)$.
Because of \mbox{weak-*} closedness of $\partial\func(f)$ and \mbox{weak-*} lower semicontinuity of $\norm{\cdot}_*$, we can choose $\sg\in\partial\func(f)$ such that $\sigma_*(f)=1/\norm{\sg}_*$.
We define $\eta:=-\sigma\sg$ and compute
\begin{align*}
    \langle\eta,v\rangle=-\sigma\langle\sg,v\rangle \leq \sigma_* \norm{\sg}_* \norm{v} = \norm{v},\qquad\forall v\in\X.
\end{align*}
This proves that $\eta\in\Phi^1_\X(0)$ and hence the optimality condition~\labelcref{eq:OC} is satisfied.

To show the upper bound~\labelcref{ineq:upper_bound_alpha*}, we use convexity of $\func$ to derive for all $\sg\in\partial\func(f)$
\begin{align*}
    &\func(f) + \langle\sg,\hat{u}-f\rangle \leq \func(\hat{u}) \\
    \implies &\func(f)-\func(\hat{u}) \leq \norm{\sg}_* \norm{f-\hat{u}} \\
    \iff &\norm{\sg}_* \geq \frac{\func(f)-\func(\hat{u})}{\norm{f-\hat{u}}}, \quad\forall v\in\argmin\func, \\
    \iff &\norm{\sg}_* \geq \sup_{\hat{u}\in\argmin\func}\frac{\func(f)-\func(\hat{u})}{\norm{f-\hat{u}}}.
\end{align*}
Taking the infimum in $\sg\in\partial\func(f)$ this implies the desired estimate.
\end{proof}

\section{Extinction Time}

\begin{prop}\label{prop:extinct}
Assume that $\argmin\func\neq\emptyset$ and that there exists $\lambda_1>0$ such that $\lambda_1\norm{u-\hat{u}}\leq \func(u)-\func(\hat{u})$ for all $\hat{u}\in\argmin\func$ and $u\in\X$ 
Then it holds that
\begin{align*}
    \left(\argmin\func\right) \cap \left(\argmin_{u\in\X} \frac{1}{p}\norm{u-f}^p + \sigma \func(u)\right) \neq \emptyset
\end{align*}
if and only if $\sigma \geq \sigma_{**}(f)$, where
\begin{align}\label{eq:alpha_**}
    \sigma_{**}(f):= \inf_{\hat{u}\in\argmin\func}\inf_{\eta\in\Phi_\X^p(\hat{u}-f)} \sup_{v\in\X} \frac{\langle -\eta, v-\hat{u}\rangle}{\func(v)-\func(\hat{u})}.
\end{align}
Furthermore, one has the estimates
\begin{align}\label{ineq:lower_bd_sigma_**}
    \sigma_{**}(f) &\geq \inf_{\hat{u}\in\argmin\func}\frac{\norm{f-\hat{u}}^p}{\func(f)-\func(\hat{u})} \\
    \sigma_{**}(f) &\leq \frac{1}{\lambda_1} \inf_{\hat{u}\in\argmin\func}\norm{f-\hat{u}}^{p-1} < \infty. 
\end{align}
\end{prop}
\begin{proof}
If $\hat{u}\in\argmin\func$ solves \labelcref{eq:p-prox}, from \labelcref{eq:OC} we get $\eta\in\Phi_\X^p(\hat{u}-f)$ and $\sg\in\partial\func(\hat{u})$ such that $0=\eta+\sigma\sg$. 
This implies
\begin{align*}
    \sigma_{**}(f) \leq \sup_{v\in\X} \frac{\langle-\eta, v-\hat{u}\rangle}{\func(v)-\func(\hat{u})} = \sigma \sup_{v\in\X} \frac{\langle\sg, v-\hat{u}\rangle}{\func(v)-\func(\hat{u})} \leq \sigma.
\end{align*}
Conversely, let $\sigma \geq \sigma_{**}(f)$.
Using the assumed coercivity together with closedness of $\argmin\func$ and the duality map, one can show that the infima in~\labelcref{eq:alpha_**} are attained and we denote the minimizers by $\hat{u}\in\argmin\func$ and $\eta\in\Phi_\X^p(\hat{u}-f)$.
Defining $\sg:=-\eta/\sigma$, it holds
\begin{align*}
    \func(\hat{u}) + \langle \sg, u-\hat{u}\rangle 
    &= \func(\hat{u}) + \frac{1}{\sigma} \langle -\eta, u-\hat{u}\rangle \\
    &= \func(\hat{u}) + \frac{\func(u)-\func(\hat{u})}{\sigma}\frac{ \langle -\eta, u-\hat{u}\rangle}{\func(u)-\func(\hat{u})} \\
    &\leq \func(\hat{u}) + \frac{\func(u)-\func(\hat{u})}{\sigma_{**}(f)} 
    \sup_{v\in\X}\frac{ \langle -\eta, v-\hat{u}\rangle}{\func(v)-\func(\hat{u})} \\
    &\leq \func(u),\qquad \forall u\in\dom(\func).
\end{align*}
This shows $\sg\in\partial\func(\hat{u})$ which means that $\hat{u}$ satisfies the optimality condition~\labelcref{eq:OC}.

Next we prove the bounds on $\sigma_{**}(f)$. 
By choosing $u=f$ in the supremum it holds
\begin{align*}
    \sigma_{**}(f) \geq \inf_{\hat{u}\in\argmin\func} \inf_{\eta\in\Phi_\X^p(\hat{u}-f)} \frac{\langle-\eta,f-\hat{u}\rangle}{\func(f)-\func(\hat{u})} = \inf_{\hat{u}\in\argmin\func}\frac{\norm{f-\hat{u}}^p}{\func(f)-\func(\hat{u})}.
\end{align*}
Using the coercivity of $\func$ one obtains
\begin{align*}
    \sigma_{**}(f) &\leq \frac{1}{\lambda_1}\inf_{\hat{u}\in\argmin\func}\inf_{\eta\in\Phi_\X^p(\hat{u}-f)}\sup_{v\in\X} \frac{\norm{\eta}_*\norm{v-\hat{u}}}{\norm{v-\hat{u}}} \\
    &= \frac{1}{\lambda_1} \inf_{\hat{u}\in\argmin\func}\norm{f-\hat{u}}^{p-1} 
    < \infty.
\end{align*}
\end{proof}

\section{Remaining Proofs}

\begin{proof}[Proof of \cref{prop:ground_states}]
Since the duality map is never empty we can choose $\eta\in\Phi_\X^p(\hat{u}_*-u_*)$ and define $\sg:=-\lambda_p\eta$.
We claim that $\sg\in\partial\func(u_*)$, which would conclude the proof.
To see this we compute
\begin{align*}
      \func(u_*)+\langle\sg,u-u_*\rangle 
    &= \func(u_*)+\lambda_p\langle\eta,u_*-u\rangle \\
    &\leq \func(u_*)+\lambda_p\left[\frac{1}{p}\norm{\hat{u}_*-u}^p-\frac{1}{p}\norm{\hat{u}_*-u_*}^p\right] \\
    &\leq \func(u_*) + (\func(u)-\func(\hat{u}_*)) - (\func(u_*)-\func(\hat{u}_*)) \\
    &= \func(u),\qquad \forall u\in\X,
\end{align*}
where we used \cref{prop:duality_proximal} and \labelcref{eq:ground_states}.
\end{proof}

\begin{proof}[Proof of \cref{prop:angle}]
By the definition of $u^{k+1/2}$ it holds
\begin{align*}
    \frac{1}{p}\norm{u^{k+1/2}-u^k}^p + \sigma^k\func(u^{k+1/2})\leq 
    \frac{1}{p}\norm{u-u^k}^p + \sigma^k\func(u),\quad\forall u\in\X.
\end{align*}
Choosing $u=\norm{u^{k+1/2}}u^k$ and reordering the resulting inequality yields
\begin{align*}
    &\phantom{\leq}\;\frac{1}{p}\frac{1}{\sigma^k}\left( \norm{u^{k+1/2}-u^k}^p - \left|\norm{u^{k+1/2}}-1\right|^p\right) \\
    &\leq \func\left(\norm{u^{k+1/2}}u^k\right) - \func(u^{k+1/2}) \\
    &\leq \norm{u^{k+1/2}}^\alpha \left(\func(u^k) - \func(u^{k+1})\right) \\
    &\leq C\left(\func(u^k) - \func(u^{k+1})\right),
\end{align*}
where we used that $u^{k+1/2}$ converges and is therefore bounded.
Summing both sides and using that $k\mapsto\func(u^k)$ is non-increasing, one obtains that the left hand side is summable.
Since by the triangle inequality it is non-negative, it therefore converges to zero.
The reformulation for Hilbert norms and $p=2$ is straightforward by expanding the square.
\end{proof}

Before we can prove the convergence statement of the proximal power method, we need a continuity property of the $p$-proximal operator, which relies on the compactness \cref{ass:compactness}. 
The proof works precisely as the one by \citet{bungert2020nonlinear}.

\begin{lemma}\label{lem:continuity_T}
Let \cref{ass:compactness} be fulfilled.
Let $v_n,v\in\X$ such that $\norm{v_n-v}\to 0$ as $n\to\infty$ and let $u_n\in \pprox_{\sigma(v_n)\func}(v_n)$ for $n\in\N$, where the regularization parameters are chosen in such a way that $\sigma(v_n)\to \sigma$ as $n\to\infty$. Then there exists $u\in\X$ such that (up to a subsequence) it holds $\norm{u_n-u}\to 0$ and $u\in \pprox_{\sigma\func}(v)$.
\end{lemma}

\begin{proof}[Proof of \cref{thm:cvgc_power_method}]
First we note that the step sizes $\sigma^k$ converge.
For the constant parameter rule this is trivial.
For the variable rule, we observe that according to \cref{prop:decrease_func} 
$$\sigma^k=\frac{c}{\func(u^k)}\leq \frac{c}{\func(u^{k+1})}=\sigma^{k+1},$$
hence the step sizes are an increasing sequence.
Furthermore, by \labelcref{eq:coercivity} the step sizes are bounded from above and hence converge to some $\sigma^*>0$.

From \cref{prop:decrease_func} we infer that the sequence $(u^k)_{k\in\N}$ is relatively compact, meaning that up to a subsequence $\norm{u^k-u^*}\to 0$ for some element $u^*\in\X$ with $\norm{u^*}=1$. 
After another subsequence refinement \cref{lem:continuity_T} implies that also the sequence $u^{k+1/2}$ can be assumed to converge to an element $\tilde{u}$ which satisfies $\tilde{u}\in \pprox_{\sigma^*}(u^*)$. 

For the constant parameter rule it holds $\sigma^*=\sigma(u^*)$ and for the variable rule because of lower semicontinuity:
\begin{align*}
    \sigma^*=\lim_{k\to\infty}\sigma^k\leq \frac{c}{\liminf_{k\to\infty}\func(u^k)} \leq \frac{c}{\func(u^*)}= \sigma(u^*).
\end{align*}
Since furthermore $\sigma(u^*)<\sigma_{**}(u^*)$ by definition (cf.~\labelcref{ineq:lower_bd_sigma_**}), we infer from \cref{prop:extinct} that $\tilde{u}\neq 0$.

From the scheme \labelcref{eq:power_it_prox} and the convergence of $u^k$ and $u^{k+1/2}$ it now follows that
\begin{align*}
u^*=\frac{\tilde{u}}{\norm{\tilde{u}}},\quad\tilde{u}\in \pprox_{\sigma^*}(u^*).
\end{align*}
If we define $\mu:=\norm{\pprox_{\sigma^*}(u^*)}>0$ this can be reordered to~\labelcref{eq:eigenvec_prox}.

It remains to show that $\mu\leq 1$ for $p=1$ and $\mu<1$ for $p>1$.
We start with the case $p=1$. 
In this case \labelcref{eq:eigenvec_prox} is equivalent to
\begin{align*}
    \norm{\mu u^*-u^*} + \sigma^* \func(\mu u^*) \leq \norm{v-u^*} + \sigma^* \func(v),\quad\forall v\in\X.
\end{align*}
Choosing $v=u^*$ then yields
\begin{align*}
    |\mu-1|\norm{u^*} + \sigma \mu \func(u^*) \leq \sigma^* \func(u^*).
\end{align*}
If we now assume that $\mu>1$ we obtain
\begin{align*}
    (\mu-1)\norm{u^*} \leq 0
\end{align*}
which is a contradiction to $u^*\neq 0$.
Hence, we have shown that $\mu\leq 1$ if $p=1$.

In the case $p>1$ we note that the optimality conditions for~\labelcref{eq:eigenvec_prox} read
\begin{align*}
    0&\in\Phi^p_\X((\mu-1)u^*) + \sigma^*\partial\func(\mu u^*) \\
    &=|\mu-1|^{p-2}(\mu-1)\Phi_\X^p(u^*) + \sigma^* \mu^{\alpha-1} \partial\func(u^*),
\end{align*}
where we used that $\Phi_\X^p$ and $\partial\func$ are homogeneous with degree $p-1$ and $\alpha-1$, respectively (cf.~\cref{prop:duality_proximal} and \citet{bungert2019asymptotic}).
Using the properties of the duality map and the subdifferential of absolutely $p$-homogeneous functionals this implies
\begin{align*}
    0 = |\mu-1|^{p-2}(\mu-1)\norm{u^*}^p + \sigma^* \mu^{\alpha-1} \func(u^*).
\end{align*}
If we assume that $\mu\geq 1$, this equality implies $u^*=0$ which is a contradiction.
\end{proof}

\begin{proof}[Proof of \cref{thm:cvgc_to_subdiff_ev}]
The statement follows from setting $\lambda=(1-\mu)|1-\mu|^{p-2}/(\sigma^*\mu^{p-1})\geq 0$ where $\mu$ is as in the previous proof.
\end{proof}

\end{appendices}

\end{document}